\documentclass{article}
\usepackage{amssymb}
\usepackage{amsmath}
\usepackage{hyperref}
\usepackage{graphicx}
\usepackage{float}
\usepackage{amsfonts}
\usepackage{xcolor}
\usepackage{soul}
\usepackage{authblk}

\newcommand{\keywords}[1]{%
  \vspace{0.4cm}
  \noindent\textbf{Keywords:} #1
}

\newtheorem{theorem}{Theorem}
\newtheorem{corollary}[theorem]{Corollary}

\newenvironment{proof}[1][Proof]{\noindent\textbf{#1.} }{\ \rule{0.5em}{0.5em}}

\title{Integral representations of high-order derivatives of Bessel
functions with respect to the order
}

\author[1]{J.L. Gonz\'{a}lez Santander}
\author[2]{David A. Harrington}
\author[3]{Angelo Melino}

\affil[1]{\small Department of Mathematics, University of Oviedo\\
        \small C/ Leopoldo Calvo Sotelo 18, Oviedo, 33007, Asturias, Spain\\
        \small \texttt{gonzalezmarjuan@uniovi.es}}
\affil[2]{\small Chemistry Department,  University of Victoria\\
        \small Box 1700 STN CSC,  Victoria, V8W 2Y2, British Columbia, Canada\\
        \small \texttt{dharr@uvic.ca}}
\affil[3]{\small Economics Department, University of Toronto\\
        \small 150 St. George Street, Toronto, M5S 3G7, Ontario, Canada\\
        \small \texttt{angelo.melino@utoronto.ca}}

\begin{document}

\maketitle

\begin{abstract}
We derive integral representations for the $n$-th derivative with respect to the order of Bessel functions of the first, second and third kinds, as well as the modified Bessel function and the Macdonald function. Based on these integral representations, we have developed a highly-efficient algorithm for the numerical evaluation of these derivatives with respect to the order, applicable to any complex argument and any real order.
\end{abstract}


\section{Introduction}

Although Bessel functions are widely used in several areas of applied mathematics, there is surprisingly little literature about evaluation of higher-order derivatives of these functions with respect to the order. Here we derive integral representations for $n$-th derivatives with respect to order for Bessel functions, modified Bessel functions and the related Hankel functions. In these integrals, $n$ is typically an exponent of the integration variable, and so numerical evaluation of these integrals should not be significantly more difficult for, say, $n=200$ compared to $n=2$. This is an advantage over some series methods in which large $n$ requires significantly more terms for accurate evaluation.

The rest of the introduction defines the quantities of interest and reviews the relevant literature. In Section \ref{Section: Integral representations}, the integral representations are derived for each of the cases, and in Section \ref{Section: Numerical results}, numerical tests are given to demonstrate that the method is effective over a wide range of orders and arguments for these Bessel functions. Both Mathematica and Maple codes for the derivatives are provided as supplementary material.

Bessel functions are named after F. W. Bessel (1784--1846), who was the first to study them systematically, although important special cases had already been obtained approximately a century earlier by scientists of the stature of Euler, Lagrange, Fourier, and Poisson \cite{Dutka}. In fact, Bessel functions first appear in Daniel Bernoulli's investigation of the small oscillations of a uniform heavy flexible chain \cite[Chap. 4. Sect. 37]{Koronev}. The ubiquity of Bessel functions in many physical applications is probably due to the fact that they arise naturally in boundary value problems of potential theory for cylindrical domains \cite[Chap.6]{Lebedev}. However, applications go far beyond physical problems and appear in many other fields. For example, in probability theory we find the von Mises distribution, which involves the modified Bessel function \cite[Sect. 3.3.6]{Fisher}. Although Bessel functions have been extensively studied since their introduction \cite{Watson}, successive derivatives with respect to the order have only been studied recently.

To review the literature on this topic, we will focus on the expressions for the first derivative of the Bessel function of the first kind with respect to the order. Similar expressions for other Bessel functions are found in the literature cited below. Recall that the Bessel function of the first kind is defined for $z\in \mathbb{C}$ as \cite[Eq. 10.2.2]{NIST}:
\begin{equation}
J_{\nu }\left( z\right) :=\left( \frac{z}{2}\right) ^{\nu
}\sum_{k=0}^{\infty }\frac{\left( -1\right) ^{k}\,\left( z/2\right) ^{2k}}{k!\,\Gamma \left( \nu +k+1\right) }.  \label{Jnu_def}
\end{equation}
In the existing literature, we found some expressions in series form, such as the following one for $\nu \neq -1,-2,\ldots $ \cite[Eq. 10.15.1]{NIST}:
\begin{equation}
\frac{\partial J_{\nu }\left( z\right) }{\partial \nu }=J_{\nu }\left(
z\right) \log \left( \frac{z}{2}\right) \,-\left( \frac{z}{2}\right) ^{\nu
}\sum_{k=0}^{\infty }\left( -1\right) ^{k}\frac{\psi \left( \nu +k+1\right)
\,\left( z/2\right) ^{2k}}{k!\,\Gamma \left( \nu +k+1\right) },
\label{D1_Bessel_J_series}
\end{equation}
where $\psi \left( z\right) $ denotes the digamma function \cite[Eq. 1.3.1]{Lebedev}:
\[
\psi \left( z\right) :=\frac{\Gamma ^{\prime }\left( z\right) }{\Gamma
\left( z\right) },
\]
and $\Gamma \left( z\right) $ denotes the gamma function, defined for $\Re
\left( z\right) >0$ as \cite[Eq. 5.2.1]{NIST}:%
\begin{equation}
\Gamma \left( z\right) :=\int_{0}^{\infty }t^{z-1}e^{-t}dt.
\label{Gamma_def}
\end{equation}
Also, we find the following expression in series form for $\nu \neq -1,-2,\ldots $ \cite[Eq. 3.1.28]{Apelblat}:
\begin{eqnarray}
\frac{\partial J_{\nu }\left( z\right) }{\partial \nu } &=&\left[ \log
\left( \frac{z}{2}\right) -\psi \left( 1+\nu \right) \right] \,J_{\nu
}\left( z\right)   \label{D1_Bessel_J} \\
&&+\frac{\left( z/2\right) ^{\nu }}{\Gamma \left( \nu +1\right) }%
\sum_{k=1}^{\infty }\frac{\,\left( z/2\right) ^{k}\,J_{k}\left( z\right) }{%
\left( k-1\right) !\,\left( k+\nu \right) ^{2}},  \nonumber
\end{eqnarray}
For integer order $\nu =0,\pm 1,\pm 2,\ldots $, we find the following finite sum representation \cite[Sect. 3.2.3]{Oberhettinger}, \cite{ReflectionJL}:
\begin{equation}
\left. \frac{\partial J_{\nu }\left( z\right) }{\partial \nu }\right\vert
_{\nu =\pm m}=\left( \pm 1\right) ^{m}\left[ \frac{\pi }{2}Y_{m}\left(
z\right) \pm \frac{m!}{2}\sum_{k=0}^{m-1}\frac{\,J_{k}\left( z\right) }{%
k!\left( m-k\right) }\left( \frac{z}{2}\right) ^{k-m}\right] ,
\label{D1_Bessel_J_integrer}
\end{equation}
where $Y_{\nu }\left( z\right) $ denotes the Bessel function of the second kind, defined for $z\in \mathbb{C}$ and $\nu \notin \mathbb{Z}$ as \cite[Eqs. 5.4.5-6]{Lebedev}:
\begin{equation}
Y_{\nu }\left( z\right) :=\frac{J_{\nu }\left( z\right) \cos \nu \pi
-J_{-\nu }\left( z\right) }{\sin \nu \pi },  \label{Ynu_def}
\end{equation}
and for $\nu =n\in \mathbb{Z}$ as
\[
Y_{n}\left( z\right) =\lim_{\nu \rightarrow n}Y_{\nu }\left( z\right) .
\]

Integral representations are another approach. For example, van der Pol gave the following formula for $z\in
\mathbb{C}$ and $\nu >0$ \cite{VanderPol}:
\begin{equation}
\frac{\partial J_{\nu }\left( z\right) }{\partial \nu }=\frac{1}{2}%
\int_{0}^{z}\mathrm{Ji}_{0}\left( z-t\right) \,\left[ J_{\nu -1}\left(
z\right) -J_{\nu +1}\left( z\right) \right] \,dt,  \label{VanderPol_formula}
\end{equation}
where the integral Bessel function of order $\nu $ is defined as \cite[Eq.
2.1.62]{Apelblat}:
\[
\mathrm{Ji}_{\nu }\left( z\right) :=-\int_{z}^{\infty }\frac{J_{\nu }\left(
t\right) }{t}\,dt.
\]
For $\nu =0$, we have the following representation \cite{Humbert}:\
\[
\mathrm{Ji}_{0}\left( z\right) =\gamma +\log \left( \frac{z}{2}\right) -\frac{z^{2}}{8}\,_{2}F_{3}\left(
\begin{array}{c}
1,1 \\
2,2,2
\end{array}
;-\frac{z^{2}}{4}\right) ,
\]
where $\gamma =0.577216\ldots $ denotes the Euler-Mascheroni constant, and $_{p}F_{q}$ denotes the generalized hypergeometric function. Recall that for $p\leq q$, the following series converges for $z\in \mathbb{C}$ and defines an entire function called the generalized hypergeometric function
\cite[Eq. 16.2.1]{NIST}:
\begin{equation}
_{p}F_{q}\left(
\begin{array}{c}
a_{1},\ldots ,a_{p} \\
b_{1},\ldots ,b_{q}
\end{array}
;z\right) =\sum_{k=0}^{\infty }\frac{\left( a_{1}\right) _{k}\cdots \left(a_{p}\right) _{k}}{\left( b_{1}\right) _{k}\cdots \left( b_{q}\right) _{k}}
\frac{z^{k}}{k!},  \label{pFq_def}
\end{equation}
where $\left( x\right) _{k}$ denotes the Pochhammer symbol \cite[Eq. 5.2.4-5]{NIST}:
\begin{equation}
\left( x\right) _{k}=\left\{
\begin{array}{ll}
x\left( x+1\right) \left( x+2\right) \cdots \left( x+k-1\right) , & \left(
x\right) _{0}=1, \\
\displaystyle
\frac{\Gamma \left( x+k\right) }{\Gamma \left( x\right) },\quad  & x\neq
0,-1,-2,\ldots
\end{array}
\right.   \label{Pochhammer_def}
\end{equation}
Another integral representation is given in \cite{ApelblatPaper}, valid for $\Re \left( z\right) >0$,
\begin{equation}
\frac{\partial J_{\nu }\left( z\right) }{\partial \nu }=\pi \nu
\int_{0}^{\pi /2}\tan t\ Y_{0}\left( z\sin ^{2}t\right) \,J_{\nu }\left(
z\cos ^{2}t\right) \,dt.  \label{DJnu_int_Apelblat}
\end{equation}

In addition to integral representations, we find closed-form expressions for $\partial J_{\nu }\left( z\right)/\partial \nu $. For example, Brychov \cite{BrychovNew} showed that for $\left\vert \arg z\right\vert <\pi $ and $\nu
\notin \mathbb{Z}$,
\begin{eqnarray}
&&\frac{\partial J_{\nu }\left( z\right) }{\partial \nu }
\label{D_J_nu_Brychov} \\
&=&\frac{\pi \left[ Y_{\nu }\left( z\right) -\cot \pi \nu \ J_{\nu }\left(
z\right) \right] }{2\Gamma ^{2}\left( \nu +1\right) }\left( \frac{z}{2}%
\right) ^{2\nu }\,_{2}F_{3}\left( \left.
\begin{array}{c}
\nu ,\nu +\frac{1}{2} \\
\nu +1,\nu +1,2\nu +1
\end{array}
\right\vert -z^{2}\right)   \nonumber \\
&&+J_{\nu }\left( z\right) \left[ \frac{1}{2\nu }-\psi \left( \nu +1\right)
+\log \left( \frac{z}{2}\right)
\begin{array}{c}
\displaystyle
\\
\displaystyle
\end{array}
\right.   \nonumber \\
&&\quad +\left. \frac{z^{2}}{4\left( \nu ^{2}-1\right) }\,_{3}F_{4}\left(
\left.
\begin{array}{c}
1,1,\frac{3}{2} \\
2,2,2-\nu ,2+\nu
\end{array}
\right\vert -z^{2}\right) \right] ,  \nonumber
\end{eqnarray}
Another closed-form expression valid for $\nu >0$ and$\,\Re \left( z\right) >0$ is derived in \cite{BesselJL},
\begin{eqnarray}
\frac{\partial J_{\nu }\left( z\right) }{\partial \nu } &=&\frac{\pi }{2}
\left[ Y_{\nu }\left( z\right) \frac{\left( z/2\right) ^{2\nu }}{\,\Gamma
^{2}\left( \nu +1\right) }\,_{2}F_{3}\left( \left.
\begin{array}{c}
\nu ,1/2+\nu  \\
2\nu +1,\nu +1,\nu +1
\end{array}
\right\vert -z^{2}\right) \right.   \label{DJnu_Meijer} \\
&&-\left. \frac{\nu J_{\nu }\left( z\right) }{\sqrt{\pi }}
G_{2,4}^{3,0}\left( z^{2}\left\vert
\begin{array}{c}
1/2,1 \\
0,0,\nu ,-\nu
\end{array}
\right. \right) \right] ,\quad .  \nonumber
\end{eqnarray}
where \ $G_{p,q}^{m,n}$ denotes the Meijer-$G$\ function. This function is usually defined by the following Mellin-Barnes integral representation \cite[Eq. 16.17.1]{NIST}:
\begin{eqnarray}
&&G_{p,q}^{m,n}\left( z\left\vert
\begin{array}{c}
a_{1},\ldots ,a_{p} \\
b_{1},\ldots ,b_{q}
\end{array}
\right. \right)   \label{Meijer_G_def} \\
&=&\frac{1}{2\pi i}\int_{L}\frac{\prod_{\ell =1}^{m}\Gamma \left( b_{\ell
}-s\right) \prod_{\ell =1}^{n}\Gamma \left( 1-a_{\ell }+s\right) }{%
\prod_{\ell =m}^{q-1}\Gamma \left( 1-b_{\ell +1}+s\right) \prod_{\ell
=n}^{p-1}\Gamma \left( a_{\ell +1}-s\right) }z^{s}ds,  \nonumber
\end{eqnarray}
where the integration path separates the poles of the factors $\Gamma \left(b_{\ell }-s\right) $ from those of the factors $\Gamma \left( 1-a_{\ell}+s\right) $. Here, $m$ and $n$ are integers such that $0\leq m\leq q$ and $0\leq n\leq p$, and none of $a_{k}-b_{j}$ is a positive integer when $1\leq k\leq n$ and $1\leq j\leq m$.

This literature overview shows that there is no simple expression for the derivative of the Bessel functions with respect to the order valid for all complex arguments $z\in \mathbb{C}$\ and real orders $\nu \in \mathbb{R}$. Furthermore, the expressions given above do not easily generalize for the $n$-th order derivative. This is the case of the integral representations in Eqs. (\ref{VanderPol_formula})\ and (\ref{DJnu_int_Apelblat})\ because the integrand itself contains the Bessel function of the first kind. Regarding closed-form expressions, some formulas are given in \cite{BrychovNew} for the second and third order derivatives, but they are quite complicated and do not seem to be generalizable for higher-order derivatives. However, the series approach is different and Eq. (\ref{D1_Bessel_J_series})\ has been generalized for the $n$-th order derivative in \cite{LiQin}, obtaining for $z\in \mathbb{C}$ and $\nu \neq -1,-2,\ldots $
\begin{eqnarray}
&&\frac{\partial ^{n}J_{\nu }\left( z\right) }{\partial \nu ^{n}}
\label{Dn_Jnu_no_negative_integer} \\
&=&\left( \frac{z}{2}\right) ^{\nu }\sum_{m=0}^{\infty }\frac{\left(
-z^{2}/4\right) ^{m}}{m!\,\Gamma \left( m+\nu +1\right) }\sum_{k=0}^{n}%
\binom{n}{k}\log ^{n-k}\left( \frac{z}{2}\right) \,\mathcal{H}_{k}\left(
m+\nu +1\right) ,  \nonumber
\end{eqnarray}
where $\mathcal{H}_{k}\left( \alpha \right) $ is calculated by recursion:\
\begin{eqnarray}
\mathcal{H}_{0}\left( \alpha \right)  &=&1,  \label{Recursion_H_Li_Qin} \\
\mathcal{H}_{k}\left( \alpha \right)  &=&-\sum_{i=0}^{k-1}\binom{k-1}{i}%
\,\psi ^{\left( i\right) }\left( \alpha \right) \,\mathcal{H}_{k-i-1}\left(
\alpha \right) ,\quad \left( k\geq 1\right) ,  \nonumber
\end{eqnarray}
and $\psi ^{\left( k\right) }\left( z\right) $ denotes the polygamma function \cite[Sect. 5.15]{NIST}, i.e., the $k$-th derivative of the digamma function. A second series expression with a recursive function is also presented in \cite{LiQin}. Despite the fact that Eq. (\ref{Dn_Jnu_no_negative_integer})\ reduces the computation of the $n$-th derivative with respect to the order to a compact formula, the recursions such as Eq.  (\ref{Recursion_H_Li_Qin}) given in \cite{LiQin} indicate that the complexity increases with $n$. Note also that these expressions do not apply for negative integer orders. Furthermore, in \cite{LiQin}, the expressions for the $n$-th derivative with respect to the order of other Bessel functions are based on Eq. (\ref{Dn_Jnu_no_negative_integer}), and the computation time is expected to increase sharply with $n$. In the present study, we bypass these drawbacks using some special
integral representations of Bessel functions.

Below, we summarize the integral representations of Bessel functions found in the literature, on which the derivations in the next section are based.

For $\Re \left( z\right) >0$, the Schl\"{a}fli's integral representation of $J_{\nu }\left( z\right) $ and $Y_{\nu }\left( z\right) $ are given respectively by \cite[Eq. 5.10.8]{Lebedev}:
\begin{eqnarray}
J_{\nu }\left( z\right)  &=&\frac{1}{\pi }\int_{0}^{\pi }\cos \left( z\sin
t-\nu t\right) \,dt  \label{Jnu_integral} \\
&-&\frac{\sin \nu \pi }{\pi }\int_{0}^{\infty }e^{-z\sinh t-\nu t}\,dt,
\nonumber
\end{eqnarray}
and \cite[Eq. 10.9.7]{NIST}:
\begin{eqnarray}
Y_{\nu }\left( z\right)  &=&\frac{1}{\pi }\int_{0}^{\pi }\sin \left( z\sin
t-\nu t\right) \,dt  \label{Ynu_integral} \\
&&-\frac{1}{\pi }\int_{0}^{\infty }\left( e^{\nu t}+e^{-\nu t}\cos \nu \pi
\right) \,e^{-z\sinh t}\,dt.  \nonumber
\end{eqnarray}

The Hankel functions are defined as \cite[Eq. 5.6.1]{Lebedev}:
\begin{eqnarray}
H_{\nu }^{\left( 1\right) }\left( z\right) &:=&J_{\nu }\left( z\right)
+i\,Y_{\nu }\left( z\right) ,  \label{H(1)(z)_def} \\
H_{\nu }^{\left( 2\right) }\left( z\right) &:=&J_{\nu }\left( z\right)
-i\,Y_{\nu }\left( z\right) .  \label{H(2)(z)_def}
\end{eqnarray}
The integral representation of $H_{\nu }^{\left( 1\right) }\left( z\right)$ for $\Im \left( z\right) >0$ is \cite[Eq. 10.9.10]{NIST}:
\begin{equation}
H_{\nu }^{\left( 1\right) }\left( z\right) =\frac{e^{-i\pi \nu /2}}{i\,\pi }
\int_{-\infty }^{+\infty }e^{i\,z\cosh t-\nu t}\,dt,  \label{H(1)_integral}
\end{equation}
and the integral representation of $H_{\nu }^{\left( 2\right) }\left(
z\right) $ for $\Im \left( z\right) <0$ is \cite[Eq. 10.9.11]{NIST}:
\begin{equation}
H_{\nu }^{\left( 2\right) }\left( z\right) =-\frac{e^{i\pi \nu /2}}{i\,\pi }
\int_{-\infty }^{+\infty }e^{-i\,z\cosh t-\nu t}\,dt.  \label{H(2)_integral}
\end{equation}

The modified Bessel function of the first kind, hereinafter referred to simply as the modified Bessel function, is defined for $z\in \mathbb{C}$ as \cite[Eq. 10.25.2]{NIST}:
\begin{equation}
I_{\nu }\left( z\right) :=\left( \frac{z}{2}\right) ^{\nu
}\sum_{k=0}^{\infty }\frac{\,\left( z/2\right) ^{2k}}{k!\,\Gamma \left( \nu
+k+1\right) },  \label{Inu_def}
\end{equation}
and for $\Re \left( z\right) >0$\ its integral representation is given by \cite[Eq. 10.32.4]{NIST}:
\begin{eqnarray}
I_{\nu }\left( z\right)  &=&\frac{1}{\pi }\int_{0}^{\pi }\,e^{z\cos t}\cos
\nu t\,dt  \label{I_nu_integral} \\
&-&\frac{\sin \nu \pi }{\pi }\int_{0}^{\infty }e^{-z\cosh t-\nu t}\,dt.
\nonumber
\end{eqnarray}

For $z\in \mathbb{C}$ and $\nu \notin \mathbb{Z}$, the modified Bessel function of the second kind, which we will refer to as the Macdonald function, is defined as \cite[Eqs. 5.7.2-3]{Lebedev}:
\begin{equation}
K_{\nu }\left( z\right) :=\frac{\pi }{2}\frac{I_{-\nu }\left( z\right)
-I_{\nu }\left( z\right) }{\sin \nu \pi },  \label{K_nu_def}
\end{equation}%
and for $\nu =n\in \mathbb{Z}$ as
\[
K_{n}\left( z\right) =\lim_{\nu \rightarrow n}K_{\nu }\left( z\right) .
\]%
For $\Re \left( z\right) >0$, we have the following integral representation
\cite[Eq. 5.10.23]{Lebedev}:%
\begin{equation}
K_{\nu }\left( z\right) =\frac{1}{2}\int_{-\infty }^{+\infty }e^{\nu
t-z\cosh t}\,dt.  \label{K_nu_integral}
\end{equation}

Throughout this paper, we consider $n=0,1,2,\ldots $ $z\in \mathbb{C} $, and $\nu \in \mathbb{R}$.

\section{Integral representations \label{Section: Integral representations}}

\subsection{Bessel function of the first kind \label{Section: BesselJ}}

\begin{theorem}
For $\Re \left( z\right) >0$, the following formula for the $n$-th\ order
derivative holds (see \cite{ReflectionJL}):
\begin{eqnarray}
\frac{\partial ^{n}J_{\nu }\left( z\right) }{\partial \nu ^{n}} &=&\frac{1}{%
\pi }\left\{ \int_{0}^{\pi }t^{n}\cos \left( z\sin t-\nu t-\frac{\pi }{2}%
n\right) \,dt\right.  \label{DnJnuRez>0} \\
&&-\left. \int_{0}^{\infty }\Im \left[ \left( i\pi -t\right) ^{n}e^{i\pi \nu
}\right] \,e^{-z\sinh t-\nu t}\,dt\right\} .  \nonumber
\end{eqnarray}
\end{theorem}

\begin{proof}
Rewrite the integral representation given in Eq. (\ref{Jnu_integral}) as follows:
\begin{eqnarray*}
J_{\nu }\left( z\right) &=&\frac{1}{\pi }\int_{0}^{\pi }\Re \left[
e^{i\,\left( z\sin t-\nu t\right) }\right] \,dt \\
&&-\frac{1}{\pi }\int_{0}^{\infty }\Im \left[ e^{\nu \left( i\pi -t\right) }%
\right] \,e^{-z\sinh t}\,dt.
\end{eqnarray*}%
Taking the $n$-th derivative with respect to the order, we obtain
\begin{eqnarray*}
\frac{\partial ^{n}J_{\nu }\left( z\right) }{\partial \nu ^{n}} &=&\frac{1}{%
\pi }\int_{0}^{\pi }\Re \left[ \left( -it\right) ^{n}\,e^{i\,\left( z\sin
t-\nu t\right) }\right] \,dt \\
&&-\frac{1}{\pi }\int_{0}^{\infty }\Im \left[ \left( i\pi -t\right)
^{n}\,e^{\nu \left( i\pi -t\right) }\right] \,e^{-z\sinh t}\,dt.
\end{eqnarray*}%
The proof is completed by taking into account that $\left( -it\right) ^{n}=t^{n}\exp \left( -i\pi
\frac{n}{2}\right) $.
\end{proof}

\begin{theorem}
For $\Re \left( z\right) <0$, the following formula for the $n$-th\ order
derivative holds:
\begin{eqnarray}
\frac{\partial ^{n}J_{\nu }\left( z\right) }{\partial \nu ^{n}}&=&\frac{1}{2\pi }\left\{ \int_{-\pi }^{\pi }r_{+}^{n}\left( z,t\right)
\,e^{\nu \,r_{+}\left( z,t\right) +i\,z\sin t}\,dt\right.   \label{DnJnuRez<0}  \\
&&+\left. i\int_{0}^{\infty }\left[ s_{+}^{n}\left( z,t\right) \,e^{\nu
\,s_{+}\left( z,t\right) }-s_{-}^{n}\left( z,t\right) \,e^{\nu \,s_{-}\left(
z,t\right) }\right] \,e^{z\,\sinh t}\,dt\right\} ,  \notag
\end{eqnarray}
where
\begin{eqnarray}
\mathrm{a}\left( z\right) &:= &\arg z-\arg \left( -z\right) ,
\label{a(z)_def} \\
r_{\pm }\left( z,t\right) &:= &i\,\left( \mathrm{a}\left( z\right) \pm
t\right) ,  \label{r(z,t)_def} \\
s_{\pm }\left( z,t\right) &:= &-t+i\,\left( \mathrm{a}\left( z\right) \pm
\pi \right) .  \label{s(z,t)_def}
\end{eqnarray}
\end{theorem}

\begin{proof}
From the definitions of $J_{\nu }\left( z\right) $ (Eq. (\ref{Jnu_def}))\ and $\mathrm{a}\left( z\right) $ (Eq. (\ref{a(z)_def})), we have
\begin{equation}
\frac{J_{\nu }\left( z\right) }{J_{\nu }\left( -z\right) }=z^{\nu }\left(
-z\right) ^{-\nu }=e^{i\nu \,\mathrm{a}\left( z\right) }.
\label{Reflection_Jnu}
\end{equation}%
Taking into account Eq. (\ref{Reflection_Jnu}) and the integral representation given in Eq. (\ref{Jnu_integral}), we have for $\Re \left( z\right) <0$
\begin{eqnarray}
&&J_{\nu }\left( z\right)   \label{Jnu_intermedio} \\
&=&e^{i\nu \,\mathrm{a}\left( z\right) }\,J_{\nu }\left( -z\right)   \notag
\\
&=&\frac{e^{i\nu \,\mathrm{a}\left( z\right) }}{2\pi }\int_{0}^{\pi }\left[
e^{i\left( -z\sin t-\nu t\right) }+e^{-i\left( -z\sin t-\nu t\right) }\right]
\,dt  \notag \\
&&-\frac{1}{\pi }\left( \frac{e^{i\nu \pi }-e^{-i\nu \pi }}{2i}\right)
\,e^{i\nu \,\mathrm{a}\left( z\right) }\int_{0}^{\infty }e^{z\sinh t-\nu
t}\,dt.  \notag
\end{eqnarray}
According to the definitions given in Eq. (\ref{r(z,t)_def})\ and Eq. (\ref{s(z,t)_def}), rewrite Eq. (\ref{Jnu_intermedio})\ as follows:
\begin{eqnarray*}
J_{\nu }\left( z\right)  &=&\frac{1}{2\pi }\left\{ \int_{0}^{\pi }\left[
e^{\nu \,r_{+}\left( z,t\right) +i\,z\sin t}+e^{\nu \,r_{-}\left( z,t\right)
-i\,z\sin t}\right] \,dt\right.  \\
&&+\left. i\int_{0}^{\infty }\left[ e^{\nu \,s_{+}\left( z,t\right) }-e^{\nu
\,s_{-}\left( z,t\right) }\right] \,e^{z\sinh t}\,dt\right\}  \\
&=&\frac{1}{2\pi }\left\{ \int_{-\pi }^{\pi }e^{\nu \,r_{+}\left( z,t\right)
+i\,z\sin t}\,dt\right.  \\
&&+\left. i\int_{0}^{\infty }\left[ e^{\nu \,s_{+}\left( z,t\right) }-e^{\nu
\,s_{-}\left( z,t\right) }\right] \,e^{z\sinh t}\,dt\right\} .
\end{eqnarray*}%
\newline
Finally, take the $n$-th derivative with respect to the order to complete the proof.
\end{proof}

\begin{theorem}
For $x>0$, the following formula for the $n$-th\ order derivative holds:
\begin{eqnarray}
\frac{\partial ^{n}J_{\nu }\left( i\,x\right) }{\partial \nu ^{n}} &=&\frac{1%
}{2\pi }\left\{ \int_{0}^{\pi }\left[ p_{+}^{n}\left( t\right)
\,e^{p_{+}\left( t\right) \,\,\nu }+p_{-}^{n}\left( t\right)
\,e^{p_{-}\left( t\right) \,\,\nu }\right] \,e^{x\cos t}\,dt\right.
\label{DnJ(ix)} \\
&&+\left. i\int_{0}^{\infty }\left[ q_{+}^{n}\left( t\right)
\,e^{q_{+}\left( t\right) \,\nu }-q_{-}^{n}\left( t\right) \,e^{q_{-}\left(
t\right) \,\nu }\right] \,e^{-x\cosh t}\,dt\right\} ,  \nonumber
\end{eqnarray}
where we have set
\begin{eqnarray}
p_{\pm }\left( t\right) &:=&i\left( \frac{\pi }{2}\pm t\right) ,
\label{p+-(t)_def} \\
q_{\pm }\left( t\right) &:=&i\left( \frac{1}{2}\pm 1\right) \pi -t.
\label{q+-(t)_def}
\end{eqnarray}
\end{theorem}

\begin{proof}
From the definitions of $J_{\nu }\left( z\right) $ (Eq. (\ref{Jnu_def})) and $I_{\nu }\left( z\right) $ (Eq. (\ref{Inu_def})), we have
\[
\frac{J_{\nu }\left( i\,z\right) }{I_{\nu }\left( z\right) }=\left(
i\,z\right) ^{\nu }z^{-\nu }=e^{i\nu \left[ \arg \left( iz\right) -\arg z
\right] }.
\]
Thus for $x>0$
\[
J_{\nu }\left( i\,x\right) =e^{i\pi \nu /2}I_{\nu }\left( x\right) .
\]
Now, apply the integral representation of $I_{\nu }\left( z\right) $ in Eq. (\ref{I_nu_integral}) and the definitions given in Eqs. (\ref{p+-(t)_def})-(\ref{q+-(t)_def}) to obtain
\begin{eqnarray*}
J_{\nu }\left( i\,x\right) &=&\frac{1}{2\pi }\left\{ \int_{0}^{\pi }\left[
\,e^{p_{+}\left( t\right) \,\,\nu }+\,e^{p_{-}\left( t\right) \,\,\nu }%
\right] \,e^{x\cos t}\,dt\right. \\
&&+\left. i\int_{0}^{\infty }\left[ e^{q_{+}\left( t\right) \,\nu
}-\,e^{q_{-}\left( t\right) \,\nu }\right] \,e^{-x\cosh t}\,dt\right\} .
\end{eqnarray*}
Finally, take the $n$-th derivative with respect to the order to complete the proof.
\end{proof}

\begin{corollary}
For $x>0$, the following formula for the $n$-th order derivative holds:
\[
\frac{\partial ^{n}J_{\nu }\left( -i\,x\right) }{\partial \nu ^{n}}=%
\overline{\frac{\partial ^{n}J_{\nu }\left( i\,x\right) }{\partial \nu ^{n}}.}
\]
\end{corollary}

\begin{proof}
Applying the property for $z \notin \left(-\infty, 0 \right] $ \cite[Eq. 10.11.9]{NIST}:
\begin{equation}
J_{\nu }\left( \bar{z}\right) =\overline{J_{\nu }\left( z\right) },
\label{Jnu(conj(z))}
\end{equation}%
we obtain for $x>0$,
\[
J_{\nu }\left( -i\,x\right) =J_{\nu }\left( \overline{i\,x}\right) = \overline{J_{\nu }\left( i\,x\right) }.
\]
Finally, take the $n$-th derivative with respect to the order to complete the proof.
\end{proof}

\begin{theorem}
\label{Theorem: DnBesselJ_z=0} For $z=0$, we have the following cases for the $n$-th order derivative:
\[
\frac{\partial ^{n}J_{\nu }\left( 0\right) }{\partial \nu ^{n}}=\left\{
\begin{array}{cc}
0, & \nu >0, \\
\mathrm{Divergent,} & \nu \leq 0.
\end{array}
\right.
\]
\end{theorem}

\begin{proof}
According to \cite[Eq. 10.7.3]{NIST}:
\begin{equation}
J_{\nu }\left( z\right) \approx \frac{\left( z/2\right) ^{\nu }}{\Gamma
\left( 1+\nu \right) },\quad z\rightarrow 0,\ \nu \neq -1,-2,\ldots
\label{J_nu(z)_z->0}
\end{equation}
Therefore, taking the $n$-th\ order derivative, the dominant term is given by the logarithmic function as $z\rightarrow 0$, i.e.,
\[
\frac{\partial ^{n}J_{\nu }\left( z\right) }{\partial \nu ^{n}}\approx \frac{%
\log ^{n}\left( z/2\right) }{\Gamma \left( 1+\nu \right) }\,\left(
z/2\right) ^{\nu },\quad z\rightarrow 0,\ \nu \neq -1,-2,\ldots
\]
Consequently,
\[
\left. \frac{\partial ^{n}J_{\nu }\left( 0\right) }{\partial \nu ^{n}}%
\right\vert _{\nu >0}=0,
\]%
but
\[
\lim_{z\rightarrow 0}\left\vert \frac{\partial ^{n}J_{\nu }\left( z\right) }{%
\partial \nu ^{n}}\right\vert =\infty ,\quad \nu \leq 0,\ \nu \neq
-1,-2,\ldots
\]
For the case $\nu =-1,-2,\ldots $, perform the change of variable $\nu
=-m+\epsilon $ in Eq. (\ref{J_nu(z)_z->0}),
\[
J_{-m+\epsilon }\left( z\right) \approx \frac{\left( z/2\right)
^{-m+\epsilon }}{\Gamma \left( 1-m+\epsilon \right) },\quad z\rightarrow 0,
\]
and therefore for $m=1,2,\ldots $, we have
\begin{equation}
\left. \frac{\partial ^{n}J_{\nu }\left( z\right) }{\partial \nu ^{n}} \right\vert _{\nu =-m}\approx \frac{\partial ^{n}}{\partial \epsilon ^{n}}\left[ \frac{\left( z/2\right) ^{-m+\epsilon }}{\Gamma \left( 1-m+\epsilon \right) }\right] _{\epsilon =0},\quad z\rightarrow 0.  \label{DJ1}
\end{equation}
However, according to \cite[Eq. 1.1.5]{Lebedev},
\[
\Gamma \left( z\right) \approx \frac{\left( -1\right) ^{k}}{k!}\frac{1}{z+k}
,\quad z\rightarrow -k,\ \left( k=0,1,2,\ldots \right) ,
\]
and therefore
\begin{equation}
\frac{\partial ^{n}}{\partial \epsilon ^{n}}\left[ \frac{\left( z/2\right)^{-m+\epsilon }}{\Gamma \left( 1-m+\epsilon \right) }\right] \approx \left(\frac{z}{2}\right) ^{-m}\left( -1\right) ^{m-1}\left( m-1\right) !\frac{\partial ^{n}}{\partial \epsilon ^{n}}\left[ \epsilon \left( z/2\right)^{\epsilon }\right] ,\quad \epsilon \rightarrow 0.  \label{DJ2}
\end{equation}
Now, apply Leibniz's derivative formula \cite[Eq. 1.4.12]{NIST} to obtain
\begin{eqnarray}
\frac{\partial ^{n}}{\partial \epsilon ^{n}}\left[ \epsilon \left(
z/2\right) ^{\epsilon }\right] _{\epsilon =0} &=&\left[ \epsilon \left(
\frac{z}{2}\right) ^{\epsilon }\log \left( \frac{z}{2}\right) +n\left( \frac{%
z}{2}\right) ^{\epsilon }\log ^{n-1}\left( \frac{z}{2}\right) \right]
_{\epsilon =0} \\
&=&\log ^{n-1}\left( \frac{z}{2}\right) .  \label{DJ3}
\end{eqnarray}
Therefore, according to Eqs. (\ref{DJ1}), (\ref{DJ2})\ and (\ref{DJ3}), we conclude that
\[
\lim_{z\rightarrow 0}\left\vert \frac{\partial ^{n}J_{\nu }\left( z\right) }{\partial \nu ^{n}}\right\vert =\infty ,\quad \nu =-1,-2,\ldots
\]
and the proof is completed.
\end{proof}

\subsection{Bessel function of the second kind}

\begin{theorem}
For $\Re \left( z\right) >0$, the following formula for the $n$-th\ order derivative holds:
\begin{eqnarray}
\frac{\partial ^{n}Y_{\nu }\left( z\right) }{\partial \nu ^{n}} &=&\frac{1}{\pi }\left\{ \int_{0}^{\pi }t^{n}\sin \left( z\sin t-\nu t-\frac{n\pi }{2}\right) \,\,dt\right.  \label{DnYnuRez>0} \\
&&-\left. \int_{0}^{\infty }\left\{ t^{n}e^{\nu t}+\Re \left[ \left( i\pi -t\right) ^{n}e^{i\pi \nu }\right] \,e^{-\nu t}\right\} \,e^{-z\sinh t}\,dt\right\} .  \nonumber
\end{eqnarray}
\end{theorem}

\begin{proof}
Rewrite the integral representation given in Eq. (\ref{Ynu_integral})\ as
\begin{eqnarray*}
Y_{\nu }\left( z\right) &=&\frac{1}{\pi }\int_{0}^{\pi }\Im \left[e^{i\left( z\sin t-\nu t\right) }\right] \,dt \\
&&-\frac{1}{\pi }\int_{0}^{\infty }\left\{ e^{\nu t}+\Re \left[ e^{\nu \left( i\pi -t\right) }\right] \right\} \,e^{-z\sinh t}\,dt,
\end{eqnarray*}
and take the $n$-th derivative with respect to the order to obtain
\begin{eqnarray*}
\frac{\partial ^{n}Y_{\nu }\left( z\right) }{\partial \nu ^{n}} &=&\frac{1}{\pi }\int_{0}^{\pi }\Im \left[ \left( -it\right) ^{n}e^{i\left( z\sin t-\nu t\right) }\right] \,dt \\
&&-\frac{1}{\pi }\int_{0}^{\infty }\left\{ t^{n}e^{\nu t}+\Re \left[ \left( i\pi -t\right) ^{n}e^{i\pi \nu }\right] e^{-\nu t}\right\} \,e^{-z\sinh t}\,dt.
\end{eqnarray*}
The proof is completed by taking into account that $\left( -it\right) ^{n}=t^{n}\exp \left( -i\pi\frac{n}{2}\right) $.
\end{proof}

\begin{theorem}
For $\Re \left( z\right) <0$, the following formula for the $n$-th\ order
derivative holds:
\begin{eqnarray}
\frac{\partial ^{n}Y_{\nu }\left( z\right) }{\partial \nu ^{n}} &=&\frac{1}{\pi }\left\{ i\int_{0}^{\pi }\left[ P_{a}\left( \nu ,z,n,t\right) +P_{b}\left( \nu ,z,n,t\right) \right] \,dt\right.  \label{DnYnuRez<0} \\
&&-\left. \int_{0}^{\infty }\left[ Q_{a}\left( \nu ,z,n,t\right)+Q_{b}\left( \nu ,z,n,t\right) \right] \,e^{z\sinh t}dt\right\} ,  \nonumber
\end{eqnarray}
where we have defined
\begin{eqnarray*}
P_{a}\left( \nu ,z,n,t\right) &:=&\frac{1}{2}\,\left[ r_{+}^{n}\left( -z,t\right) \,e^{\nu \,r_{+}\left( -z,t\right) +i\,z\sin t}-r_{-}^{n}\left( -z,t\right) \,e^{\nu \,r_{-}\left( -z,t\right) -i\,z\sin t}\right] , \\
P_{b}\left( \nu ,z,n,t\right) &:=&\frac{\mathrm{a}\left( z\right) }{\pi }\left[ \left( t-\pi \right) ^{n}\cos \left( \nu \left( t-\pi \right) +z\sin t+n\frac{\pi }{2}\right) \right. \\
&&+\left. \left( t+\pi \right) ^{n}\cos \left( \nu \left( t+\pi \right) +z\sin t+n\frac{\pi }{2}\right) \right] , \\
Q_{a}\left( \nu ,z,n,t\right) &:=&\frac{i\,\mathrm{a}\left( z\right) }{\pi }\,\Im \left[ \left( 2\pi i-t\right) ^{n}\,e^{\left( 2\pi i-t\right) \nu }\right] , \\
Q_{b}\left( \nu ,z,n,t\right) &:=&\left( t-i\,\mathrm{a}\left( z\right)\right) ^{n}e^{\nu \,\left( t-i\,\mathrm{a}\left( z\right) \right) } \\
&&+\frac{1}{2}\,\left[ s_{+}^{n}\left( -z,t\right) \,e^{\nu \,s_{+}\left( -z,t\right) }+s_{-}^{n}\left( -z,t\right) \,e^{\nu \,s_{-}\left( -z,t\right)}\right] .
\end{eqnarray*}
\end{theorem}

\begin{proof}
For $\nu \notin \mathbb{Z}$ we have \cite{ReflectionxJL}:
\begin{equation}
J_{\nu }\left( z\right) =\frac{i\pi }{2\,\mathrm{a}\left( z\right) \cos \nu
\pi }\,\left[ Y_{\nu }\left( -z\right) -e^{i\nu \,\mathrm{a}\left( z\right)
}\,Y_{\nu }\left( z\right) \right] ,  \label{Reflection_Ynu_Jnu}
\end{equation}%
where recall that $\mathrm{a}\left( z\right) $ is given by Eq. (\ref{a(z)_def}).
Taking into account Eq. (\ref{Reflection_Jnu}), rewrite Eq. (\ref{Reflection_Ynu_Jnu})\ as
\begin{equation}
Y_{\nu }\left( z\right) =e^{-i\nu \,\mathrm{a}\left( z\right) }\,Y_{\nu
}\left( -z\right) +\frac{2i}{\pi }\,\mathrm{a}\left( z\right) \cos \nu \pi
\,J_{\nu }\left( -z\right) .  \label{Ynu(z)->Ynu(-z)}
\end{equation}
On one hand, from the integral representation of $J_{\nu }\left( z\right) $ (Eq. (\ref{Jnu_integral})), we apply the trigonometric identities
\cite[Eq. 4.21.7\&27]{NIST}:
\begin{eqnarray*}
2\cos u\cos v &=&\cos \left( u-v\right) +\cos \left( u+v\right) , \\
\sin u\cos u &=&\frac{1}{2}\sin \left( 2u\right) =\frac{1}{2}\Im \left[ e^{2iu}\right] ,
\end{eqnarray*}
to obtain
\begin{eqnarray}
&&\cos \nu \pi \,J_{\nu }\left( -z\right)  \label{cos*Jnu} \\
&=&\frac{1}{2\pi }\left\{ \int_{0}^{\pi }\left[ \cos \left( \nu \left( t-\pi
\right) +z\sin t\right) +\cos \left( \nu \left( t+\pi \right) +z\sin
t\right) \right] \,dt\right.  \nonumber \\
&&- \left. \int_{0}^{\infty }\Im \left[ e^{\left( 2\pi i-t\right) \,\nu }\right]
\,e^{z\sinh t}dt.\right\}  \nonumber
\end{eqnarray}
On the other hand, from the integral representation of $Y_{\nu }\left( z\right) $ (Eq. (\ref{Ynu_integral})), and taking into account the definitions given in Eqs. (\ref{r(z,t)_def}) and (\ref{s(z,t)_def}), we obtain
\begin{eqnarray}
&&e^{-i\nu \,\mathrm{a}\left( z\right) }\,Y_{\nu }\left( -z\right)
\label{exp*Ynu} \\
&=&\frac{1}{2\pi }\left\{ i\int_{0}^{\pi }\left[ e^{\nu \,r_{+}\left(
-z,t\right) +i\,z\sin t}-\,e^{\nu \,r_{-}\left( -z,t\right) -i\,z\sin t}\right] \,dt\right.  \nonumber \\
&&-\left. \int_{0}^{\infty }\left[ 2e^{\nu \,\left( t-i\,\mathrm{a}\left(z\right) \right) }+\,e^{\nu \,s_{+}\left( -z,t\right) }+\,e^{\nu\,s_{-}\left( -z,t\right) }\right] \,e^{z\sinh t}dt\right\} .  \nonumber
\end{eqnarray}
Substitute Eqs. (\ref{cos*Jnu})\ and (\ref{exp*Ynu})\ in Eq. (\ref{Ynu(z)->Ynu(-z)}) and take the $n$-th derivative with respect to the order to complete the proof.
\end{proof}

\begin{theorem}
For $x>0$, the following formula for the $n$-th\ order derivative holds:
\begin{eqnarray*}
\frac{\partial ^{n}Y_{\nu }\left( i\,x\right) }{\partial \nu ^{n}} &=&i\,\frac{\partial ^{n}J_{\nu }\left( i\,x\right) }{\partial \nu ^{n}} \\
&&+\frac{\left( -1\right) ^{n+1}}{\pi }\int_{-\infty }^{+\infty }\left( t+i\frac{\pi }{2}\right) ^{n}e^{-\nu \left( t+i\,\pi /2\right) -x\cosh t}\,dt,
\end{eqnarray*}
where $\frac{\partial ^{n}J_{\nu }\left( i\,x\right) }{\partial \nu ^{n}}$
is given by Eq. (\ref{DnJ(ix)}).
\end{theorem}

\begin{proof}
From the definition of $H_{\nu }^{\left( 1\right) }\left( z\right)$ (Eq. (\ref{H(1)(z)_def})) and the integral representation given in Eq. (\ref{H(1)_integral}), we have for $\Im \left( z\right) >0$
\begin{eqnarray*}
Y_{\nu }\left( z\right) &=&i\,\left[ J_{\nu }\left( z\right) -H_{\nu}^{\left( 1\right) }\left( z\right) \right] \\
&=&i\,J_{\nu }\left( z\right) -\frac{1}{\pi }\int_{-\infty }^{+\infty}e^{-\nu \left( t+i\,\pi /2\right) +i\,z\cosh t}\,dt.
\end{eqnarray*}
Let $z=i\,x$ and take the the $n$-th derivative with respect to the order to complete the proof.
\end{proof}

\begin{corollary}
For $x>0$, the following formula for the $n$-th\ order derivative holds:
\[
\frac{\partial ^{n}Y_{\nu }\left( -i\,x\right) }{\partial \nu ^{n}}=\overline{\frac{\partial ^{n}Y_{\nu }\left( i\,x\right) }{\partial \nu ^{n}}.%
}
\]
\end{corollary}

\begin{proof}
Applying the property for $z \notin \left(-\infty, 0 \right] $ \cite[Eq. 10.11.9]{NIST}:
\begin{equation}
Y_{\nu }\left( \bar{z}\right) =\overline{Y_{\nu }\left( z\right) },
\label{Ynu(conj(z))}
\end{equation}
we obtain for $x>0$,
\[
Y_{\nu }\left( -i\,x\right) =Y_{\nu }\left( \overline{i\,x}\right) =%
\overline{Y_{\nu }\left( i\,x\right) }.
\]%
Finally, take the $n$-th derivative with respect to the order to complete the proof.
\end{proof}

\begin{theorem}
\label{Theorem: DnYnu(0)} For $z=0$, the $n$-th\ order derivative of the Bessel function of the second kind is divergent.
\end{theorem}

\begin{proof}
According to \cite[Eqs. 10.7.4-5]{NIST} and \cite[Eq. 54:9:1]{Atlas}, we have the following asymptotic behaviour of the Bessel function of the second kind as $z\rightarrow 0$:
\[
Y_{\nu }\left( z\right) \approx \left\{
\begin{array}{ll}
\displaystyle
-\frac{1}{\pi }\Gamma \left( \nu \right) \left( \frac{z}{2}\right) ^{-\nu },
& \nu >0, \\
\displaystyle
\frac{2}{\pi }\left[ \log \left( \frac{z}{2}\right) +\gamma \right] , & \nu
=0, \\
\displaystyle
-\frac{1}{\pi }\Gamma \left( -\nu \right) \cos \nu \pi \left( \frac{z}{2}\right) ^{-\nu }, & \nu <0.
\end{array}
\right.
\]
It is clear that the $n$-th\ derivative with respect to the order in all the above cases produces divergent terms when $z\rightarrow 0$.
\end{proof}

\subsection{Modified Bessel function}

\begin{theorem}
For $\Re \left( z\right) >0$, the following formula for the $n$-th\ order derivative holds:
\begin{eqnarray}
\frac{\partial ^{n}I_{\nu }\left( z\right) }{\partial \nu ^{n}} &=&\frac{1}{%
\pi }\left\{ \int_{0}^{\pi }t^{n}\cos \left( \nu t+\frac{n\pi }{2}\right)
\,e^{z\cos t}\,dt\right.  \label{DnBesselIRez>0} \\
&&-\left. \int_{0}^{\infty }\Im \left[ \left( i\pi -t\right) ^{n}e^{i\pi \nu
}\right] \,e^{-z\cosh t-\nu t}\,dt\right\} .  \nonumber
\end{eqnarray}
\end{theorem}

\begin{proof}
Rewrite the integral representation given in Eq. (\ref{I_nu_integral})\ as follows
\begin{eqnarray}
I_{\nu }\left( z\right) &=&\frac{1}{\pi }\int_{0}^{\pi }e^{z\cos t}\,\Re \left[ e^{i\nu t}\right] \,dt  \label{I_nu_integral2} \\
&&- \frac{1}{2\pi i}\int_{0}^{\infty }\Im \left[ e^{\nu \left( i\pi -t\right) }\right] \,e^{-z\cosh t}\,dt,  \nonumber
\end{eqnarray}
and take the the $n$-th derivative with respect to the order to complete the proof.
\end{proof}

\begin{theorem}
For $\Re \left( z\right) <0$, the following formula for the $n$-th\ order derivative holds:
\begin{eqnarray*}
&&\frac{\partial ^{n}I_{\nu }\left( z\right) }{\partial \nu ^{n}} \\
&=&\frac{1}{2\pi }\left\{ \int_{0}^{\pi }\left[ r_{+}^{n}\left( z,t\right)\,e^{r_{+}\left( z,t\right) \,\nu }+r_{-}^{n}\left( z,t\right)\,e^{r_{-}\left( z,t\right) \,\nu }\right] \,e^{-z\cos t}\,dt\right. \\
&&+\,i\left. \int_{0}^{\infty }\left[ s_{+}^{n}\left( z,t\right)\,e^{s_{+}\left( z,t\right) \,\nu }-s_{-}^{n}\left( z,t\right)\,e^{s_{-}\left( z,t\right) \,\nu }\right] \,e^{z\cosh t}\,dt\right\} ,
\end{eqnarray*}%
where recall that $r_{\pm }\left( z,t\right) $ and $s_{\pm }\left(z,t\right) $ are defined in Eqs. (\ref{r(z,t)_def})\ and (\ref{s(z,t)_def})\
respectively.
\end{theorem}

\begin{proof}
From the definition of the modified Bessel function in Eq. (\ref{Inu_def}) and
recalling the definition of $\mathrm{a}\left( z\right) $ in Eq. (\ref{a(z)_def}), we have
\begin{equation}
\frac{I_{\nu }\left( z\right) }{I_{\nu }\left( -z\right) }=z^{\nu }\left( -z\right) ^{-\nu }=e^{i\nu \,\mathrm{a}\left( z\right) }.
\label{Reflection_I_nu}
\end{equation}
Thus, according to Eq. (\ref{Reflection_I_nu})\ and the integral representation in Eq. (\ref{I_nu_integral}), we have for $\Re \left( z\right) <0$,
\begin{eqnarray*}
I_{\nu }\left( z\right) &=&\frac{e^{i\nu \,\mathrm{a}\left( z\right) }}{\pi }\int_{0}^{\pi }\left[ e^{i\nu t}+e^{-i\nu t}\right] \,e^{-z\cos t}\,dt \\
&&- \frac{e^{i\nu \,\mathrm{a}\left( z\right) }}{2\pi i}\int_{0}^{\infty }\left[ e^{\nu \left( i\pi -t\right) }-e^{\nu \left( -i\pi -t\right) }\right]\,e^{z\cosh t}\,dt.
\end{eqnarray*}
Now, applying the definitions given in Eqs. (\ref{r(z,t)_def})-(\ref{s(z,t)_def}), we obtain
\begin{eqnarray*}
I_{\nu }\left( z\right) &=&\frac{1}{2\pi }\left\{ \int_{0}^{\pi }\left[e^{\nu \,r_{+}\left( z,t\right) }+e^{\nu \,r_{-}\left( z,t\right) }\right]\,e^{-z\cos t}\,dt\right. \\
&&+\left. i\int_{0}^{\infty }\left[ e^{s_{+}\left( z,t\right) \,\nu}-e^{s_{-}\left( z,t\right) \,\nu }\right] \,e^{z\cosh t}\,dt\right\} .
\end{eqnarray*}
Finally, take the $n$-th derivative with respect to the order to complete the proof.
\end{proof}

\begin{theorem}
For $x>0$, the following formula for the $n$-th\ order derivative holds:
\begin{eqnarray*}
&&\frac{\partial ^{n}I_{\nu }\left( i\,x\right) }{\partial \nu ^{n}} \\
&=&\frac{1}{2\pi }\left\{ \int_{0}^{\pi }\left[ p_{+}^{n}\left( t\right)
\,e^{-ix\sin t+p_{+}\left( t\right) \,\,\nu }+p_{-}^{n}\left( t\right)
\,e^{ix\sin t+p_{-}\left( t\right) \,\nu }\right] \,dt\right. \\
&&+\left. i\int_{0}^{\infty }\left[ q_{+}^{n}\left( t\right)
\,e^{q_{+}\left( t\right) \,\nu }-q_{-}^{n}\left( t\right) \,e^{q_{-}\left(
t\right) \,\nu }\right] \,e^{-x\sinh t}\,dt\right\} .
\end{eqnarray*}
\end{theorem}

\begin{proof}
From the definitions of the Bessel function of the first kind in Eq. (\ref{Jnu_def})\ and the modified Bessel function in Eq. (\ref{Inu_def}), we have
\[
\frac{I_{\nu }\left( i\,z\right) }{J_{\nu }\left( z\right) }=\left(
i\,z\right) ^{\nu }z^{-\nu }=e^{i\nu \left[ \arg \left( iz\right) -\arg z \right] }.
\]
Therefore, for $x>0$
\begin{equation}
I_{\nu }\left( i\,x\right) =e^{i\pi \nu /2}J_{\nu }\left( x\right) .
\label{Inu(ix)}
\end{equation}%
Rewrite the integral representation given in Eq. (\ref{Jnu_integral})\ as
follows:\
\begin{eqnarray*}
&&J_{\nu }\left( x\right) =\frac{1}{2\pi }\int_{0}^{\pi }\left[ e^{i\left(x\sin t-\nu t\right) }+e^{-i\left( x\sin t-\nu t\right) }\right] \,dt \\
&&-\frac{1}{2\pi i}\int_{0}^{\infty }\left[ e^{\nu \left( i\pi -t\right)}-e^{\nu \left( -i\pi -t\right) }\right] \,e^{-x\sinh t}\,dt.
\end{eqnarray*}
Now, taking into account Eq. (\ref{Inu(ix)})\ and the definitions given in Eqs. (\ref{p+-(t)_def})-(\ref{q+-(t)_def}), we obtain
\begin{eqnarray*}
I_{\nu }\left( i\,x\right) &=&\frac{1}{2\pi }\int_{0}^{\pi }\left[e^{-ix\sin t+p_{+}\left( t\right) \,\nu }+e^{ix\sin t+p_{-}\left( t\right)\,\nu }\right] \,dt \\
&-&\frac{1}{2\pi i}\int_{0}^{\infty }\left[ e^{q_{+}\left( t\right) \,\nu}-e^{q_{-}\left( t\right) \,\nu }\right] \,e^{-x\sinh t}\,dt.
\end{eqnarray*}%
Finally, take the $n$-th derivative with respect to the order to complete the proof.
\end{proof}

\begin{corollary}
For $x>0$, the following formula for the $n$-th\ order derivative holds:
\[
\frac{\partial ^{n}I_{\nu }\left( -i\,x\right) }{\partial \nu ^{n}}=\overline{\frac{\partial ^{n}I_{\nu }\left( i\,x\right) }{\partial \nu ^{n}}.}
\]
\end{corollary}

\begin{proof}
For $z \notin \left(-\infty, 0 \right] $ \cite[Eq. 10.34.7]{NIST}, we have
\begin{equation}
I_{\nu }\left( \bar{z}\right) =\overline{I_{\nu }\left( z\right) }, \label{BesselI_conjugate}
\end{equation}
and therefore for $x>0$,
\[
I_{\nu }\left( -i\,x\right) =I_{\nu }\left( \overline{i\,x}\right) =\overline{I_{\nu }\left( i\,x\right) }.
\]
Finally, take the $n$-th derivative with respect to the order to complete the proof.
\end{proof}

\begin{theorem}
For $z=0$, the following cases for the $n$-th order derivative hold:
\begin{equation}
\frac{\partial ^{n}I_{\nu }\left( 0\right) }{\partial \nu ^{n}}=\left\{
\begin{array}{cc}
0, & \nu >0, \\
\mathrm{Divergent,} & \nu \leq 0.
\end{array}
\right.  \label{Dn_Inu_z->0}
\end{equation}
\end{theorem}

\begin{proof}
According to \cite[Eq. 10.30.1]{NIST}:
\[
I_{\nu }\left( z\right) \approx \frac{\left( z/2\right) ^{\nu }}{\Gamma \left( 1+\nu \right) },\quad z\rightarrow 0,\ \nu \neq -1,-2,\ldots
\]
which is equivalent to the asymptotic behaviour of $J_{\nu }\left( z\right)$ as $z\rightarrow 0$ given in Eq. (\ref{J_nu(z)_z->0}). Therefore, repeating the
same steps as in Theorem \ref{Theorem: DnBesselJ_z=0}, we arrive at Eq. (\ref{Dn_Inu_z->0}).
\end{proof}

\subsection{Macdonald function}

\begin{theorem}
For $\Re \left( z\right) >0$, the following integral representation for the $n$-th\ order derivative holds \cite{ReflectionJL}:\
\begin{equation}
\frac{\partial ^{n}K_{\nu }\left( z\right) }{\partial \nu ^{n}}=\frac{1}{2}
\int_{-\infty }^{+\infty }t^{n}\,e^{\nu t-z\cosh t}\,dt.  \label{Dn_Knu_x>0}
\end{equation}
\end{theorem}

\begin{proof}
Take the $n$-th derivative with respect to the order in the integral representation given in Eq. (\ref{K_nu_integral}).
\end{proof}

\begin{theorem}
For $\Re \left( z\right) <0$, the following integral representation for the $n$-th\ order derivative holds:
\begin{eqnarray}
\frac{\partial ^{n}K_{\nu }\left( z\right) }{\partial \nu ^{n}}&=&\frac{1}{2}\int_{-\infty }^{+\infty }\left( t-i\,\mathrm{a}\left( z\right) \right)^{n}\,e^{\nu \left( t-i\,\mathrm{a}\left( z\right) \right) +z\cosh t}\,dt \label{Dn_Knu_x<0}\\
&&-i\,\mathrm{a}\left( z\right) \,\frac{\partial ^{n}I_{\nu }\left( -z\right) }{\partial \nu ^{n}}\,,  \notag
\end{eqnarray}
where $\frac{\partial ^{n}I_{\nu }\left( -z\right) }{\partial \nu ^{n}}$ is given by Eq. (\ref{DnBesselIRez>0}).
\end{theorem}

\begin{proof}
According to the reflection formula \cite{ReflectionxJL}:\
\[
I_{\nu }\left( z\right) =\frac{i}{\mathrm{a}\left( z\right) }\,\left[ K_{\nu}\left( z\right) \,e^{i\,\mathrm{a}\left( z\right) \,\nu }-K_{\nu }\left(-z\right) \right] ,
\]
we obtain
\[
K_{\nu }\left( z\right) =e^{-i\,\mathrm{a}\left( z\right) \,\nu }K_{\nu}\left( -z\right) -i\,\mathrm{a}\left( z\right) \,I_{\nu }\left( -z\right) .
\]
Now, apply the integral representation in Eq. (\ref{K_nu_integral}) to obtain for $\Re \left( z\right) <0$,
\[
K_{\nu }\left( z\right) =\frac{1}{2}\int_{-\infty }^{+\infty }e^{\nu \left( t-i\,\mathrm{a}\left( z\right) \right) +z\cosh t}\,dt-i\,\mathrm{a}\left( z\right) \,I_{\nu }\left( -z\right) .
\]
Finally, take the $n$-th derivative with respect to the order to complete the proof.
\end{proof}

\begin{theorem}
For $x>0$, the following formula for the $n$-th\ order derivative holds:
\begin{eqnarray*}
\frac{\partial ^{n}K_{\nu }\left( i\,x\right) }{\partial \nu ^{n}} &=&\frac{1}{2}\left\{ \int_{0}^{\infty }\left[ \left( t-i\,\frac{\pi }{2}\right)^{n}e^{\nu \left( t-i\pi /2\right) }+\tilde{q}^{n}\left( t\right) \,e^{\nu \,\tilde{q}\left( t\right) }\right] \,e^{-x\sinh t}\,dt\right. \\
&&-\left. i\int_{0}^{\pi }\tilde{p}^{n}\left( t\right) \,e^{-ix\sin t+\nu \,\tilde{p}\left( t\right) }\,dt\right\} ,
\end{eqnarray*}
where we have defined
\begin{eqnarray}
\tilde{p}\left( t\right) &:=&i\left( \frac{\pi }{2}+t\right) ,\label{p_tilde_def} \\
\tilde{q}\left( t\right) &:=&i\,\frac{\pi }{2}+t.  \label{q_tilde_def}
\end{eqnarray}
\end{theorem}

\begin{proof}
According to \cite[Eq. 10.27.11]{NIST}, we have for $x>0$
\begin{equation}
K_{\nu }\left( i\,x\right) =-\frac{\pi }{2}\left[ e^{-i\pi \nu /2}Y_{\nu}\left( x\right) +i\,e^{-i\pi \nu }I_{\nu }\left( i\,x\right) \right] .
\label{Knu(ix)}
\end{equation}
Taking into account the integral representations given in Eqs. (\ref{Ynu_integral})\ and (\ref{I_nu_integral}), as well as the definitions given in Eqs. (\ref{p+-(t)_def})-(\ref{q+-(t)_def}), we obtain
\begin{eqnarray}
&&e^{-i\pi \nu /2}Y_{\nu }\left( x\right)   \label{exp*Ynu(x)} \\
&=&\frac{1}{2\pi }\left\{ i\int_{0}^{\pi }\left[ e^{-i\,x\sin t}e^{-\nu\,p_{+}\left( -t\right) }-e^{i\,x\sin t}e^{-\nu \,p_{-}\left( -t\right) }\right] \,dt\right.   \nonumber \\
&&-\left. \int_{0}^{\infty }\left[ 2\,e^{\nu \left( t-i\pi /2\right)}+e^{-\nu \,q_{+}\left( -t\right) }+e^{-\nu \,q_{-}\left( -t\right) }\right]\,e^{-x\sinh t}dt\right\} ,  \nonumber
\end{eqnarray}
and
\begin{eqnarray}
&&i\,e^{-i\pi \nu }I_{\nu }\left( i\,x\right)   \label{exp*Inu(ix)} \\
&=&\frac{1}{2\pi }\left\{ i\int_{0}^{\pi }\left[ e^{-i\,x\sin t}e^{-\nu\,p_{+}\left( -t\right) }+e^{i\,x\sin t}e^{-\nu \,p_{-}\left( -t\right) }\right] \,dt\right.   \nonumber \\
&&+\left. \int_{0}^{\infty }\left[ e^{-\nu \,q_{+}\left( -t\right) }-e^{-\nu\,q_{-}\left( -t\right) }\right] \,e^{-x\sinh t}dt\right\} .  \nonumber
\end{eqnarray}
Substitute Eq. (\ref{exp*Ynu})\ and Eq.(\ref{exp*Inu(ix)})\ in Eq. (\ref{Knu(ix)}), taking into account that
\begin{eqnarray*}
\tilde{p}\left( t\right)  &=&-p_{+}\left( -t\right) , \\
\tilde{q}\left( t\right)  &=&-q_{-}\left( -t\right) ,
\end{eqnarray*}
to obtain
\begin{eqnarray*}
K_{\nu }\left( i\,x\right)  &=&\frac{1}{2}\left\{ \int_{0}^{\infty }\left[e^{\nu \left( t-i\pi /2\right) }+e^{\nu \,\tilde{q}\left( t\right) }\right]\,e^{-x\sinh t}dt\right.  \\
&&-\left. i\int_{0}^{\pi }e^{-i\,x\sin t}e^{\nu\,\tilde{p}\left( t\right)}\,dt\right\} .
\end{eqnarray*}
Finally, take the $n$-th derivative with respect to the order to complete the proof.
\end{proof}

\begin{corollary}
For $x>0$, the following formula for the $n$-th\ order derivative holds:
\[
\frac{\partial ^{n}K_{\nu }\left( -i\,x\right) }{\partial \nu ^{n}}=\overline{\frac{\partial ^{n}K_{\nu }\left( i\,x\right) }{\partial \nu ^{n}}.}
\]
\end{corollary}

\begin{proof}
Applying the property for $z \notin \left(-\infty, 0 \right] $ \cite[Eq. 10.34.7]{NIST}:
\begin{equation}
K_{\nu }\left( \bar{z}\right) =\overline{K_{\nu }\left( z\right) }, \label{BesselK_Conjugte}
\end{equation}
we obtain for $x>0$,
\[
K_{\nu }\left( -i\,x\right) =K_{\nu }\left( \overline{i\,x}\right) =\overline{K_{\nu }\left( i\,x\right) }.
\]%
Finally, take the the $n$-th derivative with respect to the order to complete the proof.
\end{proof}

\begin{theorem}
For $z=0$, we have the following cases for the $n$-th\ order derivative:
\[
\frac{\partial ^{n}K_{\nu }\left( 0\right) }{\partial \nu ^{n}}=\left\{
\begin{array}{cc}
\mathrm{Divergent,} & \nu \neq 0, \\
0, & \nu =0,\,n=1,3,5,\ldots \\
\mathrm{Divergent,} & \nu =0,\,n=2,4,6,\ldots
\end{array}
\right.
\]
\end{theorem}

\begin{proof}
According to \cite[Eq. 5.7.10]{Lebedev}, we have the symmetry relation:\
\begin{equation}
K_{\nu }\left( z\right) =K_{-\nu }\left( z\right) .  \label{Knu=K-nu}
\end{equation}
Therefore, taking the $n$-th derivative with respect to the order, we obtain:
\[
\frac{\partial ^{n}K_{\nu }\left( z\right) }{\partial \nu ^{n}}=\left( -1\right) ^{n}\frac{\partial ^{n}K_{-\nu }\left( z\right) }{\partial \nu ^{n}}.
\]
Consequently,
\[
\left. \frac{\partial ^{n}K_{\nu }\left( z\right) }{\partial \nu ^{n}}\right\vert _{\nu =0}=0,\quad n=1,3,5,\ldots
\]
For the other cases, we have a divergent behaviour. Indeed, according to \cite[Eq. 10.30.2]{NIST} and the symmetry property in Eq. (\ref{Knu=K-nu}), we have the following asymptotic behaviour of the Macdonald function as $z\rightarrow 0$:
\[
K_{\nu }\left( z\right) \approx \left\{
\begin{array}{ll}
\displaystyle
\frac{1}{2}\Gamma \left( \left\vert \nu \right\vert \right) \left( \frac{z}{2}\right) ^{-\left\vert \nu \right\vert }, & \nu \neq 0, \\
\displaystyle
-\log z, & \nu =0.
\end{array}
\right.
\]
It is clear that the $n$-th\ derivative with respect to the order in all the above cases produces divergent terms when $z\rightarrow 0$.
\end{proof}

\subsection{Hankel function of the first kind}

\begin{theorem}
For $\Re \left( z\right) >0$, the following formula for the $n$-th\ order derivative holds:
\begin{eqnarray*}
\frac{\partial ^{n}H_{\nu }^{\left( 1\right) }\left( z\right) }{\partial \nu^{n}} &=&\frac{1}{\pi }\left\{ \int_{0}^{\pi }t^{n}\,e^{i\left( z\sin t-\nu t-\pi n/2\right) }\,dt\right. \\
&&-\left. i\int_{0}^{\infty }\left[ t^{n}\,e^{\nu t}+\left( -i\pi -t\right)^{n}\,e^{-\nu \left( i\pi +t\right) }\right] \,e^{-z\sinh t}\,dt\right\} .
\end{eqnarray*}
\end{theorem}

\begin{proof}
According to the definition of the Hankel function of the first kind in Eq. (\ref{H(1)(z)_def}), we have
\begin{equation}
\frac{\partial ^{n}H_{\nu }^{\left( 1\right) }\left( z\right) }{\partial \nu^{n}}=\frac{\partial ^{n}J_{\nu }\left( z\right) }{\partial \nu ^{n}}+i\,\frac{\partial ^{n}Y_{\nu }\left( z\right) }{\partial \nu ^{n}}.
\label{DnH1nu=DnJnu+i*DnYnu}
\end{equation}
Substitute the integral representations given in Eqs. (\ref{DnJnuRez>0})\ and (\ref{DnYnuRez>0}), taking into account the following identities:
\begin{eqnarray*}
e^{i\,z} &=&\cos z+i\sin z, \\
-i\,\bar{z} &=&-i\,\Re \left( z\right) -\Im \left( z\right) .
\end{eqnarray*}
Simplify the result to complete the proof.
\end{proof}

\begin{corollary} \label{Hankel1_integral_Rez_positive}
For $\Re \left( z\right) >0$, the following integral representation holds:
\[
H_{\nu }^{\left( 1\right) }\left( z\right) =\frac{1}{\pi }\left\{\int_{0}^{\pi }e^{i\left( z\sin t-\nu t\right) }\,dt-i\int_{0}^{\infty }\left[ e^{\nu t}+e^{-\nu \left( i\pi +t\right) }\right] \,e^{-z\sinh t}\,dt\right\} .
\]
\end{corollary}

\begin{theorem}
For $\Re \left( z\right) <0$, the following formula for the order derivative of $n$-th order holds:
\begin{eqnarray}
\frac{\partial ^{n}H_{\nu }^{\left( 1\right) }\left( z\right) }{\partial \nu^{n}} &=&\frac{1}{\pi }\left\{ \int_{0}^{\pi }\left[ \tilde{P}_{a}\left( \nu,z,n,t\right) +\tilde{P}_{b}\left( \nu ,z,n,t\right) \right] \,dt\right.\label{DnH1Rez<0} \\
&&+\left. \int_{0}^{\infty }\left[ \tilde{Q}_{a}\left( \nu ,z,n,t\right) +\tilde{Q}_{b}\left( \nu ,z,n,t\right) \right] \,e^{z\sinh t}dt\right\} ,
\nonumber
\end{eqnarray}
where we have defined
\begin{eqnarray*}
&&\tilde{P}_{a}\left( \nu ,z,n,t\right) :=\frac{1}{2}\,\left[\,r_{+}^{n}\left( z,t\right) \,e^{\nu \,r_{+}\left( z,t\right)}-r_{+}^{n}\left( -z,t\right) \,e^{\nu \,r_{+}\left( -z,t\right) }\right]
\,e^{i\,z\sin t} \\
&&\qquad \qquad \qquad +\frac{1}{2}\,\left[ \,r_{-}^{n}\left( z,t\right)\,e^{\nu \,r_{-}\left( z,t\right) }+r_{-}^{n}\left( -z,t\right) \,e^{\nu\,r_{-}\left( -z,t\right) }\right] \,e^{-i\,z\sin t}, \\
&&\tilde{P}_{b}\left( \nu ,z,n,t\right) :=-P_{b}\left( \nu ,z,n,t\right) , \\
&&\tilde{Q}_{a}\left( \nu ,z,n,t\right) :=\frac{\,\mathrm{a}\left( z\right)}{\pi }\,\Im \left[ \left( 2\pi i-t\right) ^{n}\,e^{\left( 2\pi i-t\right)\nu }\right] -i\,\left( t-i\,\mathrm{a}\left( z\right) \right) ^{n}e^{\nu\,\left( t-i\,\mathrm{a}\left( z\right) \right) }, \\
&&\tilde{Q}_{b}\left( \nu ,z,n,t\right) :=-\Im \left[ s_{+}^{n}\left(z,t\right) \,e^{\nu \,s_{+}\left( z,t\right) }\right] \,-i\,\Re \left[s_{-}^{n}\left( z,t\right) \,e^{\nu \,s_{-}\left( z,t\right) }\right] .
\end{eqnarray*}
\end{theorem}

\begin{proof}
Substitute the integral representations given in Eqs. (\ref{DnJnuRez<0}) and (\ref{DnYnuRez<0}) into Eq. (\ref{DnH1nu=DnJnu+i*DnYnu}), taking into account the property:
\[
\overline{s_{\pm }\left( z,t\right) }=s_{\mp }\left( -z,t\right) .
\]
Simplify the result to complete the proof.
\end{proof}

\begin{corollary} \label{Hankel1_integral_Rez_negative}
For $\Re \left( z\right) <0$, the following integral representation holds:
\begin{eqnarray*}
&&H_{\nu }^{\left( 1\right) }\left( z\right) \\
&=&\frac{1}{\pi }\left\{ \int_{0}^{\pi }\Bigl[ \cos \left( \nu \left(\mathrm{a}\left( z\right) +t\right) +z\sin t\right) +i\sin \left( \nu \left(\mathrm{a}\left( z\right) -t\right) -z\sin t\right) \right. \\
&&\quad -\left. \frac{2\mathrm{a}\left( z\right) }{\pi }\cos \nu \pi \cos \left( \nu t+z\sin t\right) \right] dt \\
&&+\int_{0}^{\infty }\left[ e^{-\nu t}\left(\frac{\mathrm{a}\left( z\right)}{\pi }\sin 2\pi \nu -\sin \left( \nu \left( \mathrm{a}\left( z\right) +\pi\right) \right) -i\cos \left( \nu \left( \mathrm{a}\left( z\right) -\pi\right) \right) \right) \right. \\
&&\quad -\left. \left. i\,e^{\nu \left( t-i\,\mathrm{a}\left( z\right)\right) }\,e^{z\sinh t}\,\right] dt\right\} .
\end{eqnarray*}
\end{corollary}
\begin{proof}
Take $n=0$ in Eq. (\ref{DnH1Rez<0}) and simplify the result.
\end{proof}

\begin{theorem}
For $\Im \left( z\right) >0$, the following integral representation for the $n$-th\ order derivative holds:
\[
\frac{\partial ^{n}H_{\nu }^{\left( 1\right) }\left( z\right) }{\partial \nu^{n}}=\frac{\left( -1\right) ^{n+1}i}{\pi }\int_{-\infty }^{+\infty }\left(t+i\frac{\pi }{2}\right) ^{n}e^{i\,z\cosh t-\nu \left( t+i\pi /2\right)}\,dt.
\]
\end{theorem}

\begin{proof}
Take the $n$-th derivative with respect to the order in the integral representation given in Eq. (\ref{H(1)_integral}).
\end{proof}

\begin{theorem}
For $\Im \left( z\right) <0$, the following integral representation for the $n$-th\ order derivative holds:
\[
\frac{\partial ^{n}H_{\nu }^{\left( 1\right) }\left( z\right) }{\partial \nu^{n}}=2\,\frac{\partial ^{n}J_{\nu }\left( z\right) }{\partial \nu ^{n}}-\frac{i}{\pi }\int_{-\infty }^{+\infty }\left( i\,\frac{\pi }{2}-t\right)^{n}e^{-i\,z\cosh t+\nu \left( i\pi /2-t\right) }\,dt,
\]
where $\frac{\partial ^{n}J_{\nu }\left( z\right) }{\partial \nu ^{n}}$ can
be evaluated by the theorems of Section \ref{Section: BesselJ}.
\end{theorem}

\begin{proof}
Substitute the reflection formulas given in Eqs. (\ref{Reflection_Jnu})\ and (\ref{Ynu(z)->Ynu(-z)}) for $J_{\nu }\left( z\right) $ and $Y_{\nu }\left( z\right) $ respectively, and the definition of $H_{\nu }^{\left( 1\right)}\left( z\right) $ given in Eq. (\ref{H(1)(z)_def}), to obtain
\[
H_{\nu }^{\left( 1\right) }\left( z\right) =\left[ e^{i\nu \,\mathrm{a}\left( z\right) }-\frac{2}{\pi }\mathrm{a}\left( z\right) \cos \nu \pi \right] \,J_{\nu }\left( -z\right) +i\,e^{-i\nu \,\mathrm{a}\left( z\right)}\,Y_{\nu }\left( -z\right) .
\]
Note that $\mathrm{a}\left( z\right) =-\pi $ for $\Im \left( z\right) <0$.
Therefore, taking into account again the definition of $H_{\nu }^{\left(
1\right) }\left( z\right) $ in Eq. (\ref{H(1)(z)_def}), and the reflection formula Eq. (\ref{Reflection_Jnu}), we obtain for $\Im \left( z\right) <0$:
\[
H_{\nu }^{\left( 1\right) }\left( z\right) =2\,J_{\nu }\left( z\right)+e^{i\pi \nu }H_{\nu }^{\left( 1\right) }\left( -z\right) .
\]
Applying the integral representation given in Eq. (\ref{H(1)_integral}), we obtain
\[
H_{\nu }^{\left( 1\right) }\left( z\right) =2\,J_{\nu }\left( z\right) - \frac{i}{\pi }\int_{-\infty }^{+\infty }e^{-i\,z\cosh t+\nu \left( i\pi/2-t\right) }\,dt.
\]
Finally, take the the $n$-th derivative with respect to the order to complete the proof.
\end{proof}

\begin{theorem}
For $z=0$, the $n$-th\ order derivative of the Hankel function of the first kind is divergent.
\end{theorem}

\begin{proof}
The result follows straightforwardly from (\ref{DnH1nu=DnJnu+i*DnYnu}) and Theorem \ref{Theorem: DnYnu(0)}.
\end{proof}

\subsection{Hankel function of the second kind}

\begin{theorem}
For $z \notin \left(-\infty, 0 \right] $, the following formula for the $n$-th\ order derivative holds:
\[
\frac{\partial ^{n}H_{\nu }^{\left( 2\right) }\left( z\right) }{\partial \nu^{n}}=\overline{\frac{\partial ^{n}H_{\nu }^{\left( 1\right) }\left( \bar{z}\right) }{\partial \nu ^{n}}}.
\]
\end{theorem}

\begin{proof}
Using the following result for $z \notin \left(-\infty, 0 \right] $ \cite[Sect. 3.1]{Oberhettinger}:
\begin{equation}
H_{\nu }^{\left( 2\right) }\left( z\right) =\overline{H_{\nu }^{\left(1\right) }\left( \bar{z}\right) }, \label{Hankel_Conjugate}
\end{equation}
we take the $n$-th derivative with respect to the order to complete the proof.
\end{proof}

\begin{theorem}
For $z=0$, the $n$-th\ order derivative of the Hankel function of the second kind is divergent.
\end{theorem}

\begin{proof}
The result follows straightforwardly from Eq. (\ref{H(2)(z)_def})\ and Theorem \ref{Theorem: DnYnu(0)}.
\end{proof}

\begin{theorem}
For $x<0$, the following formula for the $n$-th\ order derivative holds:
\begin{equation*}
\frac{\partial ^{n}H_{\nu }^{\left( 2\right) }\left( x\right) }{\partial \nu
^{n}}=\frac{\partial ^{n}J_{\nu }\left( x\right) }{\partial \nu ^{n}}-i\,%
\frac{\partial ^{n}Y_{\nu }\left( x\right) }{\partial \nu ^{n}},
\end{equation*}
where $\frac{\partial ^{n}J_{\nu }\left( x\right) }{\partial \nu ^{n}}$ is
calculated using Eq. (\ref{DnJnuRez<0}) and $\frac{\partial ^{n}Y_{\nu }\left(
x\right) }{\partial \nu ^{n}}$ using Eq.(\ref{DnYnuRez<0}).
\end{theorem}

\begin{proof}
The result follows straightforwardly taking the $n$-th derivative w.r.t the order in Eq. (\ref{H(2)(z)_def}).
\end{proof}

\section{Numerical results \label{Section: Numerical results}}

In Figure \ref{Figure: Algorithm_Cnu}, we present the algorithm for the computation of the $n$-th derivative with respect to the order of any of the functions $J_{\nu }\left( z\right) $, $Y_{\nu }\left( z\right) $, $I_{\nu }\left( z\right) $ or $K_{\nu }\left( z\right) $.

\begin{figure}[htbp]
\centering
\includegraphics[width=0.9\textwidth]{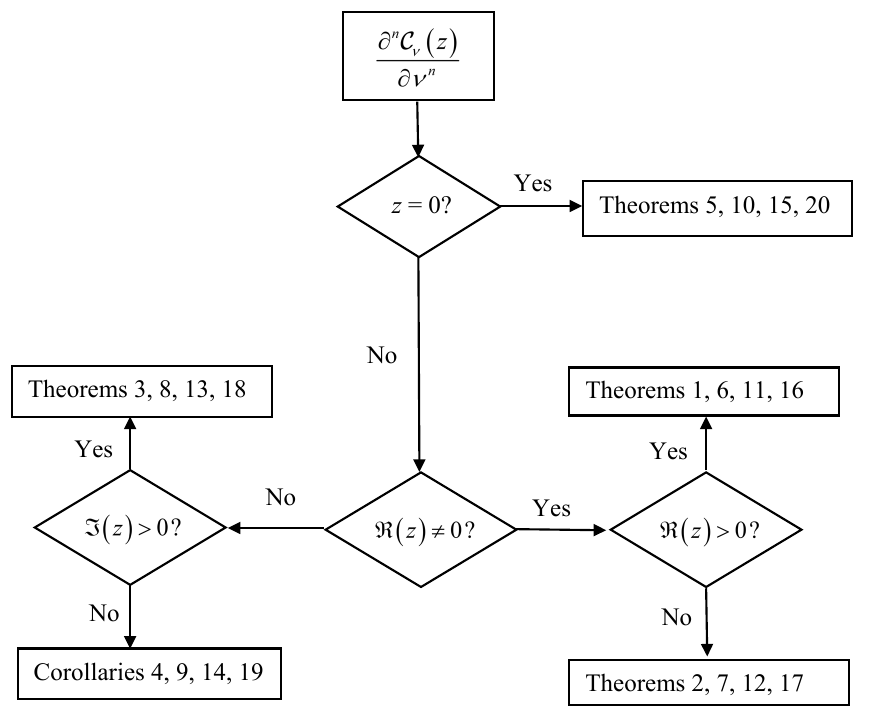}
\caption{Algorithm for the computation of the $n$-th order derivative of any of the functions $J_{\nu }\left( z\right) $, $Y_{\nu }\left( z\right) $, $I_{\nu }\left( z\right) $ or $K_{\nu }\left( z\right) $, generically denoted as $C_{\nu }\left( z\right) $.}
\label{Figure: Algorithm_Cnu}
\end{figure}

In Figure \ref{Figure: Algorithm_Hankel}, we present the algorithm for the computation of the $n$-th derivative with respect to the order of the Hankel functions $H_{\nu }^{\left( 1\right) }\left( z\right) $, $H_{\nu }^{\left(2\right) }\left( z\right) $, denoted as $H_{\nu }^{\left( k\right) }\left(
z\right) $.

\begin{figure}[htbp]
\centering
\includegraphics[width=0.95\textwidth]{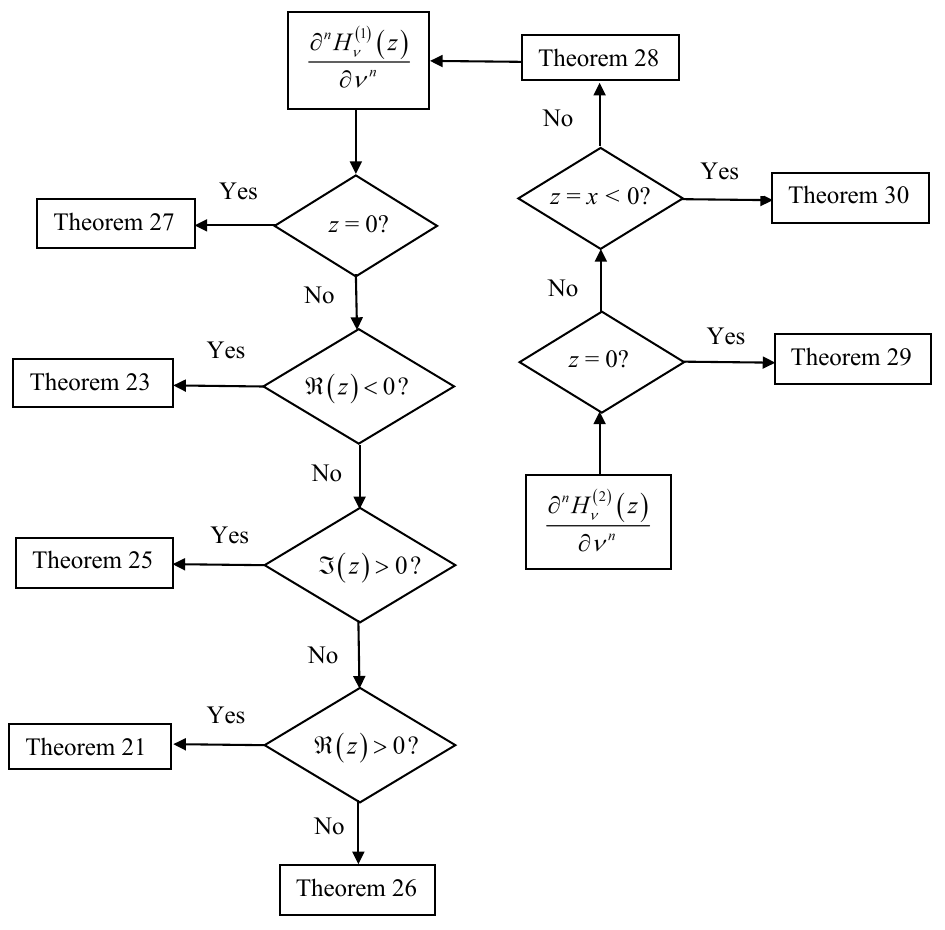}
\caption{Algorithm for the computation of the $n$-th order derivative of the Hankel functions.}
\label{Figure: Algorithm_Hankel}
\end{figure}

Note that the algorithms presented in Figures \ref{Figure: Algorithm_Cnu} and \ref{Figure: Algorithm_Hankel} prioritize the most simple integral representations for computational purposes.

As an example of these methods, Figure \ref{Figure: BesselJDerivatives} shows the first $6$ derivatives with respect to order of $J_\nu(x)$.

\begin{figure}[htbp]
\centering
\includegraphics[width=0.95\textwidth]{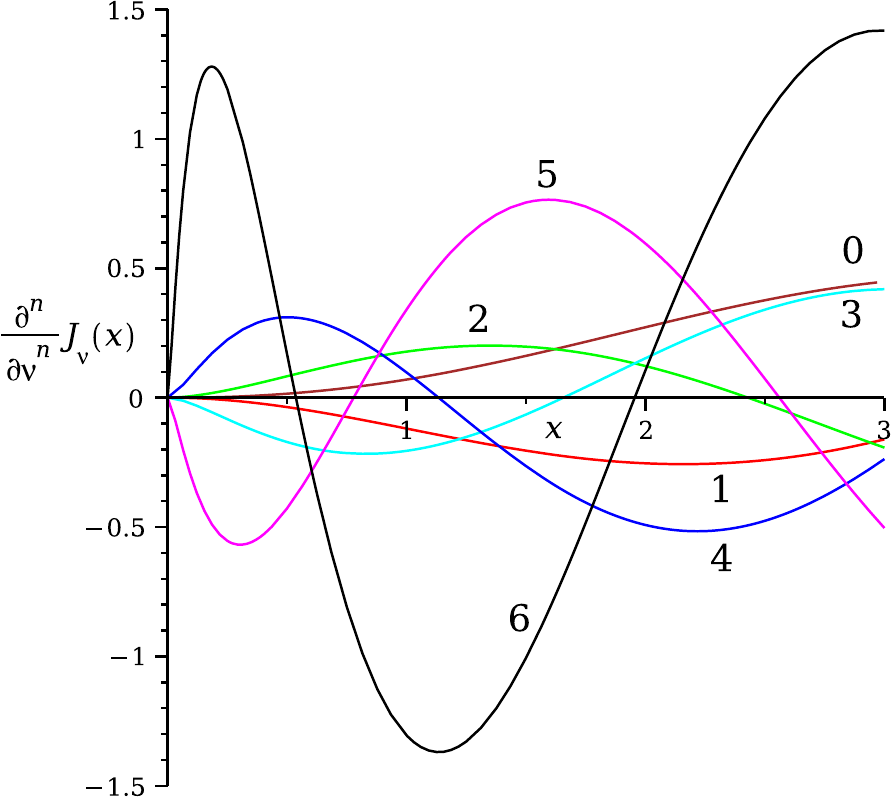}
\caption{$n$-th derivatives of $J_\nu(z)$ with respect to order evaluated at $\nu = 2.3$. Labels on curves are the value of $n$}
\label{Figure: BesselJDerivatives}
\end{figure}

The integral representations were implemented in both Mathematica and Maple. Infinite integrals were numerically evaluated by the double-exponential method \cite{Takahasi1973}, which is a built-in procedure in both languages. Finite integrals were evaluated with the same method in Maple but the default method in Mathematica.

Sample calculations at 30 digit precision will be used to illustrate the capabilities of the method. Further speed optimization may be possible but for simplicity the choice of the double-exponential method was the only non-default setting.

Results were very accurate for moderate values of $\nu$, $z$, and $n$. To quantify the accuracy, the relative error, $r$, of a function evaluated at 30 digits precision, $f_{30}$, was estimated as
\[
r = |(f_{30}-f_{40})/f_{40}|.
\]
For example, the 5th derivatives of $J_3(z),Y_3(z),I_3(z),K_3(3)$, $H_3^{(1)}(z)$ and $H_3^{(2)}(z)$ for $z \in \{\pm 2, \pm 2i, \pm 2 \pm 2i\}$, had relative errors between $5\times 10^{-31}$ and $1 \times 10^{-28}$ (between $2\times 10^{-32}$ and $4 \times 10^{-30}$ in Mathematica). These were evaluated in an average time of $1.3~ \mathrm{s}$. The average time in Mathematica was much faster, only $0.03~ \mathrm{s}$. Timings vary widely depending on the values of the parameters, reflecting the different types of integrands and integrals being evaluated.

For comparison, numerical differentiation had relative errors as high as $5 \times 10^{-23}$, even for this low $n$ value.

It is a premise of this paper that integral representations allow high-order derivatives to be easily evaluated. For example, the 100th derivatives of all six functions at $\nu=3$, $z=2$  were calculated without degraded accuracy in an average time of $2.8~ \mathrm{s}$ (Maple) or $0.1~ \mathrm{s}$ (Mathematica). The computation time did not vary significantly with $n$ until an approximate threshold value, after which computation times increased steeply or accuracy degraded. The increased time is related to the difficulty of achieving the default accuracy at the specified number of digits, and could generally be reduced by calculating at a higher number of digits.

The effects of the different integrands and integrals on the relative errors is more apparent at higher $n$, as seen in Tables \ref{Table: MMa relative errors} and \ref{Table: Maple relative errors} for $n=50$. The errors here at 30 digits precision vary widely, but in all cases the relative errors are reduced to less that $10^{-30}$ by moving to 50 digits precision.

\begin{table}[h] 
\centering
\resizebox{\textwidth}{!}{
\begin{tabular}{ccccccc}
$z$ & $\left. \frac{\partial ^{50}J_{\nu }(z)}{\partial \nu ^{50}}\right\vert _{\nu =3}$ & $\left. \frac{\partial ^{50}Y_{\nu }(z)}{\partial\nu ^{50}}\right\vert _{\nu =3}$ & $\left. \frac{\partial ^{50}I_{\nu }(z)}{\partial \nu ^{50}}\right\vert _{\nu =3}$ & $\left. \frac{\partial^{50}K_{\nu }(z)}{\partial \nu ^{50}}\right\vert _{\nu =3}$ & $\left. \frac{\partial ^{50}H_{\nu }^{(1)}(z)}{\partial \nu ^{50}}\right\vert _{\nu =3}$ & $\left. \frac{\partial ^{50}H_{\nu }^{(2)}(z)}{\partial \nu ^{50}}\right\vert _{\nu =3}$ \\ \hline
$2$ & $1\times 10^{-27}$ & $2\times 10^{-27}$ & $2\times 10^{-28}$ & $1\times 10^{-36}$ & $8\times 10^{-28}$ & $8\times 10^{-28}$ \\
$2+2i$ & $2\times 10^{-29}$ & $3\times 10^{-29}$ & $6\times 10^{-30}$ & $2\times 10^{-33}$ & $3\times 10^{-31}$ & $2\times 10^{-29}$ \\
$2-2i$ & $2\times 10^{-29}$ & $3\times 10^{-29}$ & $6\times 10^{-30}$ & $2\times 10^{-33}$ & $2\times 10^{-29}$ & $3\times 10^{-34}$ \\
$-2$ & $1\times 10^{-25}$ & $1\times 10^{-25}$ & $2\times 10^{-26}$ & $1\times 10^{-32}$ & $2\times 10^{-22}$ & $1\times 10^{-25}$ \\
$-2+2i$ & $1\times 10^{-21}$ & $1\times 10^{-21}$ & $1\times 10^{-22}$ & $7\times 10^{-27}$ & $9\times 10^{-17}$ & $1\times 10^{-21}$ \\
$-2-2i$ & $1\times 10^{-21}$ & $1\times 10^{-21}$ & $1\times 10^{-22}$ & $7\times 10^{-27}$ & $1\times 10^{-21}$ & $9\times 10^{-17}$ \\
$2i$ & $2\times 10^{-26}$ & $2\times 10^{-26}$ & $2\times 10^{-25}$ & $9\times 10^{-38}$ & $9\times 10^{-38}$ & $2\times 10^{-26}$ \\
$-2i$ & $2\times 10^{-26}$ & $1\times 10^{-26}$ & $2\times 10^{-25}$ & $9\times 10^{-38}$ & $9\times 10^{-26}$ & $9\times 10^{-38}$
\end{tabular}
}
\caption{Mathematica relative errors for 50-th derivatives of Bessel and Hankel functions at $\nu=3$ and selected $z$ values. Calculations at 30 digit precision.}\label{Table: MMa relative errors}
\end{table}

\begin{table}[h]
\centering
\resizebox{\textwidth}{!}{
\begin{tabular}{ccccccc}
$z$ & $\left. \frac{\partial ^{50}J_{\nu }(z)}{\partial \nu ^{50}}\right\vert _{\nu =3}$ & $\left. \frac{\partial ^{50}Y_{\nu }(z)}{\partial\nu ^{50}}\right\vert _{\nu =3}$ & $\left. \frac{\partial ^{50}I_{\nu }(z)}{\partial \nu ^{50}}\right\vert _{\nu =3}$ & $\left. \frac{\partial^{50}K_{\nu }(z)}{\partial \nu ^{50}}\right\vert _{\nu =3}$ & $\left. \frac{\partial ^{50}H_{\nu }^{(1)}(z)}{\partial \nu ^{50}}\right\vert _{\nu =3}$ & $\left. \frac{\partial ^{50}H_{\nu }^{(2)}(z)}{\partial \nu ^{50}}\right\vert _{\nu =3}$ \\ \hline
$2$ & $1\times 10^{-26}$ & $2\times 10^{-26}$ & $7\times 10^{-28}$ & $3\times 10^{-30}$ & $1\times 10^{-26}$ & $1\times 10^{-26}$ \\
$2+2i$ & $6\times 10^{-29}$ & $3\times 10^{-30}$ & $9\times 10^{-30}$ & $5\times 10^{-31}$ & $2\times 10^{-27}$ & $2\times 10^{-27}$ \\
$2-2i$ & $6\times 10^{-29}$ & $3\times 10^{-30}$ & $9\times 10^{-30}$ & $5\times 10^{-31}$ & $4\times 10^{-29}$ & $2\times 10^{-27}$ \\
$-2$ & $1\times 10^{-24}$ & $2\times 10^{-24}$ & $7\times 10^{-25}$ & $2\times 10^{-30}$ & $4\times 10^{-18}$ & $8\times 10^{-25}$ \\
$-2+2i$ & $1\times 10^{-21}$ & $8\times 10^{-21}$ & $2\times 10^{-21}$ & $4\times 10^{-26}$ & $8\times 10^{-13}$ & $4\times 10^{-21}$ \\
$-2-2i$ & $1\times 10^{-21}$ & $8\times 10^{-21}$ & $2\times 10^{-21}$ & $4\times 10^{-26}$ & $4\times 10^{-21}$ & $8\times 10^{-13}$ \\
$2i$ & $3\times 10^{-25}$ & $3\times 10^{-25}$ & $2\times 10^{-25}$ & $1\times 10^{-30}$ & $3\times 10^{-31}$ & $3\times 10^{-25}$ \\
$-2i$ & $3\times 10^{-25}$ & $3\times 10^{-25}$ & $2\times 10^{-25}$ & $1\times 10^{-30}$ & $3\times 10^{-25}$ & $3\times 10^{-31}$
\end{tabular}
}
\caption{Maple relative errors for 50-th derivatives of Bessel and Hankel functions at $\nu=3$ and selected $z$ values. Calculations at 30 digit precision.}\label{Table: Maple relative errors}
\end{table}

The derivatives of $K_{\nu}(z)$ were calculated very quickly, owing to the simple form of the integrals in Eq. \ref{Dn_Knu_x>0}. The $1000$-th derivative of $K_5(3)$ took only $0.9~\mathrm{s}$ (Maple) or $0.03~\mathrm{s}$ (Mathematica) to evaluate. In contrast, the more complicated integrals for $H_\nu^{(1)}(z)$ in Eq. \ref{DnH1Rez<0} were sufficiently slow that Eq. \ref{DnH1nu=DnJnu+i*DnYnu} was faster, despite requiring four integral evaluations rather than two. This faster method was adopted in the Maple code and is reflected in the timings above, but was not necessary in the much faster Mathematica code.

The theorems here apply for all real orders. Series expressions often are not valid for negative integer orders, e.g., in the work of Li and Qin \cite{LiQin}, but were easily evaluated here. For example, the 100th derivatives at $\nu=-3$ and $z=2$ were evaluated at comparable times and accuracies as for the $\nu=3$ results reported above.

\section{Conclusions \label{Section: Conclusions}}
We have obtained efficient algorithms for the numerical evaluation of the $n$-th derivative with respect to the order of Bessel functions. These algorithms are based on integral representations of Bessel functions in different regions of the complex plane.
As a by-product, we have obtained two new integral representations of the Hankel function $H_{\nu}^{(1)}(z)$ in Corollaries \ref{Hankel1_integral_Rez_positive} and \ref{Hankel1_integral_Rez_negative}.

We have numerically checked the expressions obtained throughout the paper with Mathematica and Maple. The corresponding Mathematica notebook is available at \url{https://shorturl.at/Xssbt}, and the Maple worksheet is available at the Maple Application Center \url{https://www.maplesoft.com/Applications/}.


\begin{thebibliography}{10}
\expandafter\ifx\csname url\endcsname\relax
  \def\url#1{\texttt{#1}}\fi
\expandafter\ifx\csname urlprefix\endcsname\relax\def\urlprefix{URL }\fi
\expandafter\ifx\csname href\endcsname\relax
  \def\href#1#2{#2} \def\path#1{#1}\fi

\bibitem{Dutka}
J.~Dutka, On the early history of {B}essel functions, Arch. Hist. Exact Sci.
  49~(2) (1995) 105--134.
\newblock \href {https://doi.org/10.1007/BF00376544}
  {\path{doi:10.1007/BF00376544}}.

\bibitem{Koronev}
B.~G. Korenev, Bessel functions and their applications, CRC Press, 2002.

\bibitem{Lebedev}
N.~N. Lebedev, Special {F}unctions and their {A}pplications, Prentice-Hall,
  Inc., 1965.

\bibitem{Fisher}
N.~I. Fisher, Statistical {A}nalysis of {C}ircular {D}ata, Cambridge University
  Press, 1995.

\bibitem{Watson}
G.~N. Watson, A {T}reatise on the {T}heory of {B}essel {F}unctions, Cambridge
  University Press, 1944.

\bibitem{NIST}
F.~W.~J. Olver, A.~B. Olde~Daalhuis, D.~W. Lozier, B.~I. Schneider, R.~F.
  Boisvert, C.~W. Clark, B.~R. Miller, B.~V. Saunders, H.~S. Cohl, M.~A.
  McClain, \href{https://dlmf.nist.gov/}{Nist {D}igital {L}ibrary of
  {M}athematical {F}unctions}, \url{https://dlmf.nist.gov/}, Release 1.2.8 of
  2026-09-15.
\newline\urlprefix\url{https://dlmf.nist.gov/}

\bibitem{Apelblat}
A.~Apelblat, Bessel and {R}elated {F}unctions. {T}heoretical {A}spects:
  {M}athematical {O}perations with {R}espect to the {O}rder, Vol.~1, De
  Gruyter, 2020.

\bibitem{Oberhettinger}
W.~Magnus, F.~Oberhettinger, R.~P. Soni, Formulas and theorems for the special
  functions of mathematical physics, Vol.~52, Springer-Verlag, 1966.

\bibitem{ReflectionJL}
J.~L. Gonz{\'a}lez-Santander, Reflection formulas for order derivatives of
  {B}essel functions, Results Math. 78~(1) (2023) 30.
\newblock \href {https://doi.org/10.1007/s00025-022-01804-3}
  {\path{doi:10.1007/s00025-022-01804-3}}.

\bibitem{VanderPol}
B.~Van Der~Pol, On the operational solution of linear differential equations
  and an investigation of the properties of these solutions, Lond. Edinb.
  Dublin Philos. Mag. J. Sci. 8~(53) (1929) 861--898.
\newblock \href {https://doi.org/10.1080/14786441208564951}
  {\path{doi:10.1080/14786441208564951}}.

\bibitem{Humbert}
P.~Humbert, Bessel-integral functions, P. Edinburgh Math. Soc. 3~(4) (1933)
  276--285.
\newblock \href {https://doi.org/10.1017/S0013091500027358}
  {\path{doi:10.1017/S0013091500027358}}.

\bibitem{ApelblatPaper}
A.~Apelblat, N.~Kravitsky, Integral representations of derivatives and
  integrals with respect to the order of the {B}essel functions {$J_{\nu}(t)$},
  {$I_{\nu}(t)$}, the anger function {$\mathbf{J}_{\nu}(t)$} and the integral
  {B}essel function {$Ji_{\nu}(t)$}, IMA J. Appl. Math. 34~(2) (1985) 187--210.
\newblock \href {https://doi.org/10.1093/imamat/34.2.187}
  {\path{doi:10.1093/imamat/34.2.187}}.

\bibitem{BrychovNew}
Y.~A. Brychkov, Higher derivatives of the {B}essel functions with respect to
  the order, Integr. Transf. Spec. F. 27~(7) (2016) 566--577.
\newblock \href {https://doi.org/10.1080/10652469.2016.1164156}
  {\path{doi:10.1080/10652469.2016.1164156}}.

\bibitem{BesselJL}
J.~L. Gonz{\'a}lez-Santander, Closed-form expressions for derivatives of
  {B}essel functions with respect to the order, J. Math. Anal. Appl. 466~(1)
  (2018) 1060--1081.
\newblock \href {https://doi.org/10.1016/j.jmaa.2018.06.043}
  {\path{doi:10.1016/j.jmaa.2018.06.043}}.

\bibitem{LiQin}
A.~Li, H.~Qin, Fast calculation of the derivatives of bessel functions with
  respect to the parameter and applications, Symmetry 15~(1) (2022) 64.
\newblock \href {https://doi.org/10.3390/sym15010064}
  {\path{doi:10.3390/sym15010064}}.

\bibitem{ReflectionxJL}
J.~L. Gonz{\'a}lez-Santander, A note on some relation formulae involving
  {B}essel functions, Integral Transforms and Special Functions 25~(12) (2014)
  992--997.
\newblock \href {https://doi.org/10.1080/10652469.2014.955642}
  {\path{doi:10.1080/10652469.2014.955642}}.

\bibitem{Atlas}
K.~B. Oldham, J.~Myland, J.~Spanier, An {A}tlas of {F}unctions: with {E}quator,
  the {A}tlas {F}unction {C}alculator, Springer, 2009.

\bibitem{Takahasi1973}
H.~Takahasi, M.~Mori, Double exponential formulas for numerical integration,
  Publications of the Research Institute for Mathematical Sciences 9~(3) (1973)
  721–741.
\newblock \href {https://doi.org/10.2977/prims/1195192451}
  {\path{doi:10.2977/prims/1195192451}}.

\end{thebibliography}
\end{document}